\documentclass[12pt]{amsproc}

\usepackage{amsmath, amssymb, amsfonts, amsthm}
\usepackage{setspace}
\usepackage{latexsym}
\usepackage{float}
\usepackage{color}
\usepackage{mathrsfs}
\usepackage{hyperref}      
\usepackage{bm}
\usepackage[margin=1in]{geometry}
\usepackage{enumerate}
\usepackage{url}
\usepackage{hyperref} 
\usepackage{colonequals}

\newtheorem{thm}{Theorem}[section]
\newtheorem{cor}[thm]{Corollary}

\newtheorem{lem}[thm]{Lemma}

\theoremstyle{definition}\newtheorem{rem}[thm]{Remark}
\theoremstyle{definition}\newtheorem{ex}[thm]{Example}

\newcommand{\Aut}{{\rm Aut}}
\newcommand{\GF}{{\rm GF}}
\newcommand{\GAP}{\textbf{{\rm GAP}}}

\newcommand{\Cay}{{\rm Cay}}

\newcommand\id{{\rm id}}

\newcommand{\GR}{{\rm GR}}

\newcommand{\calD}{\mathcal{D}}
\newcommand{\calE}{\mathcal{E}}

\newcommand{\calG}{\mathcal{G}}

\newcommand{\calQ}{\mathcal{Q}}
\newcommand{\calR}{\mathcal{R}}

\newcommand{\calU}{\mathcal{U}}

\renewcommand\le{\leqslant}
\renewcommand\ge{\geqslant}

\newcommand{\F}{\mathbb{F}}
\newcommand{\Z}{\mathbb{Z}}

\newcommand{\tr}{{\rm tr}}

\begin{document}
 
\title[Combinatorial transfer]{Combinatorial transfer: a new method for constructing infinite families of nonabelian difference sets, partial difference sets, and relative difference sets}

\author{Eric Swartz}
\address{Department of Mathematics, William \& Mary, Williamsburg, VA 23187}
\email{easwartz@wm.edu}

\author{James A. Davis}
\address{Department of Mathematics \& Statistics, University of Richmond, VA 23173}
\email{jdavis@richmond.edu}

\author{John Polhill}
\address{Department of Mathematics, Computer Science, and Digital Forensics, Commonwealth University, Bloomsburg, PA 17815}
\email{jpolhill@commonwealthu.edu}

\author{Ken W. Smith}
\address{Huntsville, TX 77340}
\email{kenwsmith54@gmail.com}

\begin{abstract}
 For nearly a century, mathematicians have been developing techniques for constructing abelian automorphism groups of combinatorial objects, and, conversely, constructing combinatorial objects from abelian groups.  While abelian groups are a natural place to start, recent computational evidence strongly indicates that the vast majority of transitive automorphism groups of combinatorial objects are nonabelian. This observation is the guiding motivation for this paper. We propose a new method for constructing nonabelian automorphism groups of combinatorial objects, which could be called the \textit{combinatorial transfer method}, and we demonstrate its power by finding (1) the first infinite families of nonabelian Denniston partial difference sets (including nonabelian Denniston PDSs of odd order), (2) the first infinite family of Spence difference sets in groups with a Sylow 3-subgroup that is non-normal and not elementary abelian, (3) the first infinite families of McFarland difference sets in groups with a Sylow $p$-subgroup that is non-normal and is not elementary abelian, (4) new infinite families of partial difference sets in nonabelian $p$-groups with large exponent, (5) an infinite family of semiregular relative difference sets whose forbidden subgroup is nonabelian, and (6) a converse to Dillon's Dihedral Trick in the PDS setting. We hope this paper will lead to more techniques to explore this largely unexplored topic.
\end{abstract}

\maketitle

\section{Introduction} 

A subset $D$ of a finite group $G$ is called a \textit{$(v, k, \lambda)$-difference set} (DS) if $|G| = v$, $|D| = k$, and, for all $g \in G \backslash \{1\}$, we can express $g$ in exactly $\lambda$ ways as $d_1d_2^{-1}$, where $d_1, d_2 \in D$.  A subset $D$ of $G$ is called a \textit{$(v,k.\lambda, \mu)$-partial difference set} (PDS) if $|G| = v$, $|D| = k$, $g \in D$ can be expressed as $d_1d_2^{-1}$ with $d_1, d_2 \in D$ in exactly $\lambda$ different ways, and $g \in G \backslash (D \cup \{1\})$ can be expressed as $d_1d_2^{-1}$ with $d_1, d_2 \in D$ in exactly $\mu$ different ways. 

Difference sets and partial difference sets have proven to be extremely useful in applications.  Difference sets have found uses, for instance, in coding theory \cite{Xia_etal_2005, Xia_etal_2007} and imaging with coded masks \cite{Skinner_1988}; see \cite{Pott_etal_1999} for the proceedings of a NATO conference dedicated to such applications.  Similarly, partial difference sets have found applications in coding theory; see, for instance, \cite[Section 8]{Ma_1994b} and \cite{Tao_Feng_Li_2021}.  For more information about DSs and PDSs, see \cite{Moore_Pollatsek_2013} and \cite{Ma_1994b}, respectively.

The study of difference sets in (noncyclic) abelian groups was initiated by the seminal works of McFarland \cite{McFarland_1973} and Turyn \cite{Turyn_1965} in the 1960s, and we now have a deep understanding of how to construct difference sets in that context. However, computational evidence indicates that difference sets in nonabelian groups will dwarf the number of difference sets in abelian groups (as noted in \cite{Polhill_etal_2023}, of the 330159 distinct $(64, 28, 12)$-DSs, there are 105269 nonisomorphic symmetric designs, of which only 748 arise from an abelian group), and there has been relatively little done to develop techniques that will work in that context. (Beyond sporadic examples, see \cite{Davis_etal_2024, Dillon_1985, Drisko_1998, McFarland_1973, Spence_1977} for the known examples of DSs in nonabelian groups and \cite{Davis_etal_2023, Feng_He_Chen_2020, Feng_Li_2021, Polhill_etal_2023, Swartz_2015} for the known examples of PDSs in nonabelian groups.)  This paper develops a technique that has so far only been used in an ad hoc way into a more general theory. Our hope is that it will motivate further exploration in this vast largely unexplored area.

% Most known constructions of (P)DSs occur in abelian groups.  Beyond sporadic examples, there are very few known examples in nonabelian groups of DSs \cite{Davis_etal_2024, Dillon_1985, Drisko_1998, McFarland_1973, Spence_1977} and PDSs \cite{Davis_etal_2023, Feng_He_Chen_2020, Feng_Li_2021, Polhill_etal_2023, Swartz_2015}.  On the other hand, as noted in \cite{Polhill_etal_2023}, of the 330159 distinct $(64, 28, 12)$-DSs, there are 105269 nonisomorphic symmetric designs, of which only 748 arise from an abelian group.  This anecdotal evidence points in two directions: first, there is likely a vast world of \textit{genuinely nonabelian (P)DSs} which do not arise from abelian groups, and, second, that it is quite likely that examples arising from abelian groups can be modified to provide new examples in a nonabelian setting.  This paper, like \cite{Davis_etal_2023}, is dedicated to pursuing this second direction.   

A method that has been used for the construction of new (P)DSs, which could be called the \textit{combinatorial transfer method}, is the following:
\begin{itemize}
 \item[(1)] take a known example of a (P)DS $D$ in a group $G$;
 \item[(2)] construct a related combinatorial object $\Gamma$, such as a symmetric design or (possibly directed) Cayley graph;
 \item[(3)] determine the full automorphism group $\Aut(\Gamma)$ of $\Gamma$; and, finally,
 \item[(4)] search $\Aut(\Gamma)$ for subgroups with the appropriate regular action on $\Gamma$.
\end{itemize}
 
The combinatorial transfer method is remarkably effective for known, ``small'' examples, and it has yielded many new examples of (P)DSs (see, for example, \cite{Jorgensen_Klin_2003}).  On the other hand, current implementations of this process have significant drawbacks.  For one thing, the vast majority of known examples of (P)DSs come from a purely algebraic construction.  Once the related combinatorial object $\Gamma$ is constructed, we often completely forget about the algebraic construction of the (P)DS when determining the full automorphism group of $\Gamma$.  While at the very least inefficient, this also makes the method extremely ad hoc, unless the associated combinatorial object $\Gamma$ itself independently has an algebraic structure, such as being related to a bilinear or quadratic form; see, for instance \cite[Section 3]{Davis_etal_2023}, \cite{Feng_He_Chen_2020}, \cite{Feng_Li_2021}, and \cite[Theorem 4.10]{Polhill_etal_2023}.  The success of these recent constructions indicates that we should not be ignoring the algebraic structure that already exists. 
 
The goal here is to take advantage of the power of the combinatorial transfer method to create new, infinite families of examples without ``forgetting'' the original construction of the (P)DS.  The following result is our main method of combinatorial transfer.  The idea is roughly as follows: given a (P)DS $D$ in a group $G$, if $\Aut(G)_D$ is the subgroup of automorphisms of $G$ that fix $D$ setwise, we prove that a new group $\calG \le \Aut(G)_D \ltimes G$ contains a (P)DS with the same parameters as $D$ if it satisfies certain conditions.  Since many (P)DSs are constructed algebraically using (for example) cosets of subgroups or ``hyperplanes'' (when a group is viewed as a vector space), it is often possible to ``arrange'' these coset representatives and hyperplanes in such a way to guarantee the existence of such a group $\calG$.

\begin{thm}
 \label{thm:transfer}
 Let $G$ be a group containing a (P)DS $D$ and let $\Aut(G)_D$ be the subgroup of automorphisms of $G$ that stabilizes $D$ setwise.  Let $\calG \le \Aut(G)_D \ltimes G$ be a subgroup satisfying:
 \begin{itemize}
  \item[(i)] $|\calG| = |G|$,
  \item[(ii)] $1 \times X = (1 \times G) \cap \calG$ is a normal subgroup of $G$ and $\calG$, and 
  \item[(iii)] the action of $\calG$ on $[G:X]$, the set of right cosets of $X$ in $G$, given by
  \[ \left(Xh\right)^{(\phi,g)} \colonequals X^\phi h^\phi g = Xh^\phi g\]
  is transitive on $[G:X]$.
 \end{itemize}
 Then, $\calG$ contains a (P)DS $\calD$ with the same parameters as $D$.
\end{thm}

We prove a similar result for relative difference sets.  Recall that $R$ is an \textit{$(m, u, k , \lambda)$-relative difference set} (RDS) in $G$ (relative to the subgroup $U \le G$ of order $u$) provided $|R| = k$, $|G| = mu$, and the set 
\[ \{ r_1r_2^{-1} : r_1 \neq r_2 \in R\}\]
contains each element of $G \backslash U$ exactly $\lambda$ times and contains no element of $U$.  The subgroup $U$ is called the \textit{forbidden subgroup} with respect to $R$.  Much like DSs and PDSs, RDSs have been useful in applications, such as the construction of mutually unbiased bases in $\mathbb{C}^n$ \cite{Godsil_Roy_2009}, which have applications to quantum information processing.

\begin{thm}
 \label{thm:RDStransfer}
 Let $G$ be a group containing an RDS $R$ with forbidden subgroup $U$ and let $\Aut(G)_R$ be the subgroup of automorphisms of $G$ that stabilizes $R$ setwise.  Let $\calG \le \Aut(G)_R \ltimes G$ be a subgroup satisfying:
 \begin{itemize}
  \item[(i)] $|\calG| = |G|$,
  \item[(ii)] $1 \times X \colonequals (1 \times G) \cap \calG$ is a normal subgroup of $G$ and $\calG$, and
  \item[(iii)] the action of $\calG$ on $[G:X]$, the set of right cosets of $X$ in $G$, given by
  \[ \left(Xh\right)^{(\phi,g)} \colonequals X^\phi h^\phi g = Xh^\phi g\]
  is transitive on $[G:X]$.
 \end{itemize}
 Then, $\calG$ contains a RDS $\calR$ with the same parameters as $R$ with respect to the forbidden subgroup
 \[ \calU = \{(\phi, g) \in \calG: g \in U\}.\]
\end{thm}

The proofs of Theorems \ref{thm:transfer} and \ref{thm:RDStransfer} are given in Section \ref{sect:transfer}.  We are then able to use Theorems these results to prove the existence of DSs, PDSs, and RDSs in many nonabelian settings, all of which are new (with the exception of some constructed in Sections \ref{sect:pgroups} and \ref{sect:McFarland}), including the following:

\begin{itemize}
 \item[(1)]  We prove a converse to Dillon's Dihedral Trick \cite[Section 4, Theorem]{Dillon_1985} in the PDS setting: whenever there exists a regular PDS (respectively, reversible DS) in an abelian group $G$ with index $2$ subgroup $X$, there exists a regular PDS (respectively, reversible DS) in the generalized dihedral extension of $X$.  In particular, as a corollary, we prove that the existence of a reversible DS in an abelian group $G$ with a subgroup of index $2$ implies the existence of a DS with the same parameters in any abelian group containing $X$ as a subgroup of index $2$ (Section \ref{sect:Dillon}).
 \item[(2)]  We prove that, whenever $C_{p^n} \times C_{p^n}$ contains a regular PDS, the nonabelian group isomorphic to $C_{p^n} \rtimes_{p^{n-1} + 1} C_{p^n}$ contains a regular PDS with the same parameters (Section \ref{sect:pgroups}).
 \item[(3)] We prove the existence of an infinite family of DSs with Spence parameters in groups with a non-normal, nonabelian Sylow $3$-subgroup, the first such family (Section \ref{sect:Spence}).  
 \item[(4)] We prove the existence of two infinite families of nonabelian $2$-groups containing PDSs with Denniston parameters in nonabelian $2$-groups, the first such examples, as well as the first construction of an infinite family of PDSs with Denniston parameters in nonabelian groups of order $p^{3m}$, where $p$ is an odd prime (Section \ref{sect:Denniston}). 
 \item[(5)] We construct new examples of DSs with McFarland parameters in nonabelian groups that have non-normal, nonabelian Sylow $p$-groups, the first such examples (Section \ref{sect:McFarland}).  
 \item[(6)] We construct new examples of RDSs in nonabelian $2$-groups, including examples where the forbidden subgroup itself is nonabelian (Section \ref{sect:RDS}). 
\end{itemize}

Finally, in Section \ref{sect:future}, we end the paper with several questions and open problems.

\section{Proofs of Theorems \ref{thm:transfer} and \ref{thm:RDStransfer}}
\label{sect:transfer}

This section is dedicated to the proofs of Theorem \ref{thm:transfer} and \ref{thm:RDStransfer} (which are quite similar).  

\begin{proof}[Proof of Theorem \ref{thm:transfer}]
We will assume that the group $G$ contains a $(v,k,\lambda)$-DS (respectively, $(v,k,\lambda, \mu)$-PDS) $D$.  We can construct a (possibly directed, possibly containing loops) Cayley graph $\Gamma \colonequals \Cay(G,D)$, and we may view $G$ as a subgroup of $A \colonequals \Aut(\Gamma)$, where $G$ acts regularly on the vertices of $\Gamma$; that is, $|V(\Gamma)| = |G|$, and, for (base vertex) $\alpha \in V(\Gamma)$, $\alpha^G = V(\Gamma)$.  (See, for example, \cite[10.14 Proposition]{Beth_Jungnickel_Lenz_1999} for the equivalence between the existence of a regular PDS $D$ in $G$ and a \textit{strongly regular Cayley graph} $\Cay(G,D)$.)

We define $A_\alpha$ to be the stabilizer of $\alpha$ in $A$; note that $A_\alpha \cap G = \{1\}$.  By the Frattini Argument, we have
\[ A = A_\alpha G;\]
indeed, since $G$ is transitive on $V(\Gamma)$, for any $a \in A$, there exists a unique $g \in G$ such that $\alpha^g = \alpha^a$.  Thus, $\alpha^{ag^{-1}} = \alpha$, meaning $ag^{-1} \in A_\alpha$ and $a \in A_\alpha G$, as desired. 

Since $\alpha$ is our base vertex of $\Gamma$, we may identify the vertices in $V(\Gamma)$ with elements of $G$ in the natural way: if $\beta = \alpha^g$ for $g \in G$, then we identify $\beta$ with $g$.  Using this identification, we may identify $\Aut(G)_D$ naturally with a subgroup of $A_\alpha$: for $\phi \in \Aut(G)_D$, we define \[(\alpha^g)^\phi \colonequals \alpha^{g^\phi}.\]
Indeed, $\alpha^g \sim \alpha^h$ if and only if $gh^{-1} \in D$ (by the definition of our Cayley graph) if and only $(gh^{-1})^\phi = g^\phi \left( h^\phi \right)^{-1}\in D$ (since $\phi \in \Aut(G)_D$) if and only if $\alpha^{g^\phi} \sim \alpha^{h^\phi}$ (by the definition of our Cayley graph), and so defining $(\alpha^g)^\phi \colonequals \alpha^{g^\phi}$ does identify the elements of $\Aut(G)_D$ with elements of $A$.  Furthermore, for any $\phi \in \Aut(G)_D$, since $1_G^\phi = 1_G$, we have
\[ \alpha^\phi = (\alpha^{1_G})^\phi = \alpha^{1_G^\phi} = \alpha^{1_G} = \alpha,\]
and so $\Aut(G)_D$ is in fact identified with a subgroup of $A_\alpha$.  This discussion shows that we may indeed identify $\Aut(G)_D \ltimes G$ as a subgroup of $A$ in a natural way and will henceforth view $\Aut(G)_D \ltimes G \le A$.

Now, let $\calG$ be a subgroup of $\Aut(G)_D \ltimes G$ with $|\calG| = |G|$.  Since $|\calG| = |V(\Gamma)|$, if we can show that $\calG$ is transitive on $V(\Gamma)$, then we can identify $\Gamma$ as a (possibly directed, possibly containing loops) Cayley graph $\Cay(\calG, \calD)$, where
\[ \calD \colonequals \{h \in \calG : \alpha^h \sim \alpha \}.\] 
Note that, since $\alpha^{(\phi,g)} = (\alpha^\phi)^g = \alpha^g$ for all $\phi \in \Aut(G)_D$, $g \in G$, it follows that
\[ \calD = \{(\phi,d) \in \calG : d \in D \}.\]
Since the defining properties of being a DS (respectively, PDS) can be expressed entirely in terms of (directed) walks on the (possibly directed, possibly containing loops) graph $\Gamma$, this implies that $\calD$ is a DS (respectively, PDS) of $\calG$ with the same parameters as $D$.

Let $1 \times X \colonequals (1 \times G) \cap \calG$, and assume $1 \times X$ is normal in both $1 \times G$ and $\calG$.     In particular, for any $(\phi, g) \in \calG$ and any $(1,x) \in 1 \times X$, we have 
\[ (\phi,g)^{-1}(1,x)(\phi, g) = (1, g^{-1}x^\phi g) \in X.\]
Because $X \lhd G$, this implies that $X$ is $\phi$-invariant for any $\phi$ such that $(\phi,g) \in \calG$.  Now, $X \le G$, so $X$ acts semiregularly on $V(\Gamma)$, i.e., $X$ has $|G:X|$ orbits of size $|X|$ on $V(\Gamma)$.  Since $1 \times X \lhd \calG$, $\calG$ has a faithful action on the orbits of $X$ on $V(\Gamma)$: for $\beta \in V(\Gamma)$ and $(\phi,g) \in \calG$,
\[ (\beta^X)^{(\phi,g)} = \beta^{X(\phi, g)} = \left( \beta^{(\phi,g)}\right)^X.\]

Finally, if we assume the natural action of $\calG$ is transitive on $[G:X]$, then (recalling the identification of $G$ with $V(\Gamma)$) $\calG$ will indeed be transitive on $V(\Gamma)$: let $\beta \in V(\Gamma)$.  Then, $\beta \in \alpha^{Xg}$ for some $g \in G$, and, by assumption $X^{(\phi,h)} = Xg$ for some $(\phi,h) \in \calG$, so $\alpha^{(\phi,h)} \in \alpha^{Xg}$.  Since $1 \times X \le \calG$ is transitive on $\alpha^{Xg} = \left(\alpha^g\right)^X$ (since $X$ is normal in both $G$ and $\calG$), there exists an element $(\psi, y) \in \calG$ such that $\alpha^{(\psi, y)} = \beta$, i.e., $\calG$ is transitive on $V(\Gamma)$, as desired.
\end{proof}
 
We now prove Theorem \ref{thm:RDStransfer}; we omit the details that apply mutatis mutandis from Theorem \ref{thm:transfer}.

\begin{proof}[Proof of Theorem \ref{thm:RDStransfer}]
 Let $\Gamma = \Cay(G,R)$.  As in the proof of Theorem \ref{thm:transfer}, if $\calG \le \Aut(G)_R \ltimes G$ has $|\calG| = |G|$ and $\calG$ is transitive on the vertices of $\Gamma$, we can then express $\Gamma$ as $\Cay(\calG, \calR)$, where
 \[ \calR \colonequals \{r \in \calG: \alpha^r \sim \alpha\}\]
 for some base vertex $\alpha \in \Gamma$ identified with $1 \in G$; indeed, proceeding exactly as in the proof of Theorem \ref{thm:transfer}, $\calG$ will indeed be transitive on $V(\Gamma)$. 
 
 Define 
 \[ \calU \colonequals \{w \in \calG: \alpha^w \in \Upsilon\},\]  
 where $\Upsilon \subseteq V(\Gamma)$ is set of vertices whose (directed) distance from $\alpha$ is either zero or more than two.  Note that
 \[ U = \{g \in G: \alpha^g \in \Upsilon\},\]
 so
 \[ \calU = \{(\phi, g) \in \calG: g \in U\}.\]
 Since $|\calG| = |G|$ and  $\calG$ is transitive on $V(\Gamma)$, we have $|\calU| = |U|$, and there is a unique element $(\phi, w) \in \calU$ for each $w \in U$.  Moreover, for each $\phi \in \Aut(G)_R$, since $R$ is an RDS with unique forbidden subgroup $U$, we must have $U^\phi = U$, and, hence, for any $(\phi_1, w_1), (\phi_2, w_2) \in \calU$, we have
 \[ (\phi_1, w_1)(\phi_2, w_2) = (\phi_1\phi_2, w_1^{\phi_2} w_2) \in \calU, \]
 implying that $\calU \le \calG$ since $U$ is a finite subgroup of $G$. Since products $r_1r_2^{-1}$, $r_1 \neq r_2 \in R$, correspond to directed walks of length two starting at $\alpha$, this will imply that $\calR$ is a RDS of $\calG$ with the same parameters as $D$.
\end{proof}

\section{A partial converse to Dillon's Dihedral Trick} 
\label{sect:Dillon}

Part of the inspiration of this paper is Dillon's Dihedral Trick (see \cite[Section 7.3]{Moore_Pollatsek_2013} for an excellent account of this result); we want to be able to determine immediately the existence of a (P)DS in one group based on the existence of one in another group.

\begin{thm}[Dillon's Dihedral Trick]\cite[Section 4, Theorem]{Dillon_1985}
\label{thm:DillonDihedral}
 Let $H$ be an abelian group, and let $G$ be the generalized dihedral extension of $H$; i.e., $G = \langle H, g \rangle$, $g^2 = 1$, and $ghg = h^{-1}$ for all $h \in H$.  If $G$ contains a DS, then so does every abelian group which contains $H$ as a subgroup of index $2$.
\end{thm}

The proof of this result relies on calculations in the group ring $\Z[G]$.  We know that a full converse to Dillon's Dihedral Trick does not hold; for example, $C_8 \times C_2$ contains a $(16,6,2)$-DS but the dihedral group of order $16$ does not.  On the other hand, Theorem \ref{thm:transfer} can be used to prove a converse to Dillon's Dihedral Trick that holds for regular PDSs (as well as \textit{reversible} DSs).

\begin{thm}
 \label{thm:DillonDihedralConverse}
 Let $X$ be an abelian group, and let $G$ be any abelian group containing $X$ as a subgroup of index $2$.  If $G$ contains a regular PDS (respectively, reversible DS) $D$, then so does the generalized dihedral extension $\calG$ of $X$.
\end{thm}

\begin{proof}
 Let $G = \langle X, y \rangle$, where $y^2 \in X$.  Let $\phi$ be the inversion automorphism, i.e., $g^{\phi} = g^{-1}$ for all $g \in G$.  Since $D$ is a regular PDS (respectively, reversible DS), $D$ is closed under inversion, and hence $\phi \in \Aut(G)_D$; moreover, $X^\phi = X$.  Define
 \[ \calG \colonequals \langle (1,x), (\phi, y): x \in X \rangle.\]
 Note that \[(\phi,y)^2 = (\phi^2, y^\phi y) = (1, y^{-1}y) = (1,1),\]
 and so $1 \times X$ is an index $2$ subgroup of $\calG$, showing that $|\calG| = |G|$ and $1 \times X$ is normal in $\calG$.  Furthermore, $G = X \cup Xy$, and 
 \[ X^{(\phi, y)} = X^{\phi}y = Xy.\]
 By Theorem \ref{thm:transfer}, $\calG$ contains a PDS (or reversible DS) with the same parameters as $D$.
 
 Finally, it remains to show that $\calG$ is isomorphic to the generalized dihedral extension of $X$.  Indeed, for any $(1,x) \in 1 \times X \cong X$, we have
 \[ (\phi, y)(1,x)(\phi, y) = (\phi^2, y^\phi x^\phi y) = (1, y^{-1}x^{-1}y) = (1, x^{-1}), \]
 proving the result.
\end{proof}

\begin{cor}
 \label{cor:Dilloniff}
 Let $G$ be an abelian group containing a reversible DS $D$, and let $X$ be a subgroup of index $2$ in $G$.  Then, every abelian group containing $X$ as a subgroup of index $2$ contains a DS with the same parameters as $D$.
\end{cor}

\begin{proof}
 Let $G$ be an abelian group containing a reversible DS $D$, and let $X$ be a subgroup of index $2$ in $G$.  By Theorem \ref{thm:DillonDihedralConverse}, the generalized dihedral extension $\calG$ of $X$ contains a DS with the same parameters as $D$.  By Dillon's Dihedral Trick (Theorem \ref{thm:DillonDihedral}), every abelian group containing $X$ as a subgroup of index $2$ contains a DS with the same parameters as $D$.
\end{proof}

\begin{rem}
 The DS obtained from Dillon's Dihedral Trick may not be reversible, even if the DS in the generalized dihedral extension is reversible.
\end{rem}

\begin{ex}
 Let 
 \[ G = \langle a, b, c : a^4 = b^2 = c^2 = [a,b] = [a,c] = [b,c] = 1 \rangle \cong C_4 \times C_4 \times C_2. \]
 Then, the subset 
 \[D = \{1, a, b, c, a^3, a^2 b c\} \]
 is a reversible $(16,6,2)$-DS in $G$.  If $\phi$ is the inversion map on $G$, then
 \[ \calG = \langle (1,a), (1,b), (\phi, c) \rangle \cong C_2 \ltimes (C_4 \times C_2)\]
 is isomorphic to the generalized dihedral extension of $\langle a, b \rangle \cong C_4 \times C_2$, and, by Theorem \ref{thm:DillonDihedralConverse}, $\calG$ contains a $(16,6,2)$-DS; namely, 
 \[ \calD = \{(1,1), (1,a), (1,b), (1, a^3), (\phi, c), (\phi, a^2bc) \}.\]
 (Note that $(1, a^2b)(\phi, c) = (\phi, a^2bc)$.)  To simplify notation, we let $h = (1,a)$, $x = (1,b)$, $g = (\phi, c)$, so
 \[ \calD = \{ 1, h, x, h^3, g, h^2xg\}.\]
 We remark that $\calD$ is actually a reversible DS in $\calG$.  We define $H \colonequals \langle h, x \rangle \cong C_4 \times C_2$ and note that $\calG$ is the generalized dihedral extension of $H$.  
 
 Define 
 \[K \colonequals \langle k, h, x : k^2 = h, h^4 = x^2 = [k,x] = 1 \rangle \cong C_8 \times C_2.\]
 Clearly, $K$ contains $H$ as an index $2$ subgroup.  By the proof of Dillon's Dihedral Trick (\cite[pp. 17--18]{Dillon_1985}), the subset $S$ of $K$ obtained by replacing ``$g$'' with ``$k$'' in $\calD$ will be a $(16, 6, 2)$-DS in $K$; indeed, 
 \[ S \colonequals \{1, h, x, h^3, k, h^2xk \}\]
 is a $(16, 6, 2)$-DS of $K$.  However, $k \in S$ but $k^{-1} \notin S$, so the DS $S$ obtained using Dillon's Dihedral Trick here is not reversible.  This means we cannot use Theorem \ref{thm:DillonDihedralConverse} to obtain a DS in the dihedral group of order $16$.  (In fact, the dihedral group of order $16$ does not contain a $(16,6,2)$-DS.) 
\end{ex}

\section{PDSs in nonabelian $p$-groups}
\label{sect:pgroups}

Throughout this section, let $p$ be a prime, let $n \ge 2$ be a positive integer, and define 
\[ G \colonequals \langle x, y : x^{p^n} = y^{p^n} = [x,y] = 1 \rangle \cong C_{p^n} \times C_{p^n}.\]
These groups contain PDSs related to partial congruence partitions, Paley PDSs, as well as reversible Hadamard difference sets when $p = 2$ (which are also PDSs, since the DSs are reversible).  When $p$ is odd, the regular PDSs in such groups were completely classified in \cite{Leifman_Muzychuk_2005}. When $p$ is even, a great deal is known, but the situation is more complicated; see \cite{Malandro_Smith_2020}.  

Given a PDS $D$ in the group $G$, it is relatively easy to find an automorphism of $G$ preserving $D$ using \textit{multipliers}.  If $m$ is an integer, we define
\[D^{(m)} \colonequals \{ d^m : d \in D \}. \]

\begin{lem}[Ma's Multiplier Theorem]
\label{lem:multiplier}
If $D = D^{(-1)}$ is a PDS in an abelian group $G$ and $m$ is an integer relatively prime to the order of $G$, then $D^{(m)} =D.$
\end{lem}

We assume that $D$ is a regular PDS in $G$, so that $D^{(-1)} = D$.  Define an automorphism $\phi$ of $G$ by
\[ g^\phi \colonequals g^{p^{n-1} + 1}\]
for all $g \in G$.  Since $p^{n-1} + 1$ is coprime to $p$, it is coprime to $|G|$, and, by Ma's Multiplier Theorem, $\phi$ is indeed an automorphism of $G$ fixing $D$.  Moreover, for any $g \in G$,
\[ g^{\phi^p} = g^{(p^{n-1} + 1)^p} = g^{\sum_{i = 0}^p {p \choose i} p^{i(n-1)}} = g,\]
so $\phi$ is an order $p$ automorphism of $G$.

The following result can be considered a generalization of the results in \cite[Section 4]{Davis_etal_2023}; note also that Theorem \ref{thm:transfer} allows for a considerably simpler proof of those results.  

\begin{thm}
 \label{thm:nonabpgroup}
 Let $G$ and $\phi$ be as defined above, suppose that $G$ contains a regular PDS $D$, and define
 \[ \calG \colonequals \langle (1, x), (\phi, y) \rangle.\]
 Then, $\calG$ is a nonabelian group isomorphic to $C_{p^n} \rtimes_{p^{n-1} + 1} C_{p^n}$ that contains a regular PDS with the same parameters as $D$.
\end{thm}

\begin{proof}
 First, when $p$ is odd, we note that
 \[(\phi, y)^p  = \left(\phi^p, \prod_{i = 0}^{p-1} y^{(p^{n-1} + 1)^i}\right) = \left(1, y^{\sum_{i=0}^{p-1}\limits (ip^{n-1} + 1)} \right) = \left(1, y^{\frac{p(p-1)}{2}p^{n-1} + p} \right) = (1, y^p),\]
 whereas when $p = 2$ we have 
 \[ (\phi, y)^4 = \left((\phi, y)^2\right)^2 = \left((\phi^2, y^\phi \cdot y)\right)^2 = (1, y^{2^{n-1} + 2})^2 = (1, y^4), \]
 which shows both that $(\phi, y)^p \in 1 \times G$ and $|(\phi, y)| = |y|$.  Moreover, if 
 \[ X \colonequals \langle x, y^p \rangle,\]
 then $\calG \cap (1 \times G) = 1 \times X$, which shows $|\calG| = |G|$, and $1 \times X$ is normal in $\calG$ since it is an index $p$ subgroup of a $p$-group.  Moreover, $x^\phi \in \langle x \rangle \subseteq X$ and 
 \[(y^p)^\phi = \left(y^p\right)^{p^{n-1} + 1} = y^{p^n + p} = y^p \in X,\]
 so $X$ is $\phi$-invariant.  Now, \[G = \bigcup_{i=0}^{p-1} Xy^i,\] and since  
 \[ X^{(\phi, y)} = X^\phi y = Xy\]
 and $y^p \in X$, we proceed by induction and see that
 \[ X^{(\phi,y)^i} = \left(Xy^{i-1} \right)^{(\phi,y)} = X^\phi (y^{i-1})^\phi y = X\left(y^{(p^{n-1} + 1)(i-1)} \right)y = X\left(y^{p^{n-1}}\right)^{i-1}y^i = Xy^i.\]
 Thus, $\calG$ is indeed transitive on the cosets of $X$ in $G$.  By Theorem \ref{thm:transfer}, $\calG$ contains a PDS with the same parameters as $D$.
 
 Finally, we note that 
 \[ (\phi, y)^{-1}(1,x)(\phi, y) = (\phi, y)^{-1}(\phi, x^{p^{n-1} + 1}y) = (\phi, y)^{-1}(\phi, y)(1, x^{p^{n-1} + 1}) = (1, x^{p^{n-1} + 1}),\]
 showing that $\calG$ is nonabelian and $\langle (1,x) \rangle \cong C_{p^n}$ is a normal subgroup of $\calG$, implying that $\calG \cong C_{p^n} \rtimes_{p^{n-1} + 1} C_{p^n}$ and completing the proof. 
\end{proof}

\begin{rem}
\label{rem:mult}
 Theorem \ref{thm:nonabpgroup} provides evidence that multipliers can likely be used in conjunction with Theorem \ref{thm:transfer} to construct many examples of PDSs in nonabelian groups whenever we have a corresponding PDS in an abelian group.
\end{rem}

\section{Nonabelian Spence DSs}
\label{sect:Spence}

\begin{thm}
 \label{thm:spencenonab3}  
 Let $d$ be a positive integer.  There exists a group $\calG$ with a nonabelian, non-normal Sylow $3$-subgroup that contains a DS with Spence parameters; namely, $\calG$ contains a DS $\calD$ with parameters
 \[ \left(\frac{3^{3d}(3^{3d} - 1)}{2}, \frac{3^{3d-1}(3^{3d} + 1)}{2}, \frac{3^{3d-1}(3^{3d-1} + 1)}{2} \right).\]
\end{thm}

\begin{proof}
 Let $r \colonequals (3^{3d} - 1)/2.$  We begin with the abelian group
 \[ G \colonequals \langle a_1, \dots, a_{3d}, w : a_i^3 = w^r =[a_i, a_j] = [a_i,w] = 1\rangle \cong C_3^{3d} \times C_{r}. \]
 Note that, if $\varphi$ denotes the Euler phi function, we have $\varphi(r)$ is divisible by $3$; indeed, there is an automorphism $\rho$ of $\langle w \rangle \cong C_r$ given by $\rho: w \mapsto w^{3^d}$.  Moreover, if 
 \[ H \colonequals \langle a_1, \dots, a_{3d} \rangle \cong C_3^{3d},\]
 then we may identify $H$ with the additive group of the finite field $\F = \GF(3^{3d})$, and there exists a Singer cycle $\sigma$ of $\F$ of order $2r = 3^{3d} - 1$, corresponding to multiplication by a primitive element $\theta$ of the field, fixing the identity and cyclically permuting the units.  Picking out a distinguished hyperplane (maximal subgroup) $H_0$ of $H$, we may choose the indices of the other hyperplanes in terms of the action of $\sigma$; that is, for $1 \le i \le r - 1$, we may define
 \[ H_i \colonequals H_0^{\sigma^i} = H_0 \cdot \theta^i = \{ h\cdot \theta^i : h \in H_0\}.\]
 Indeed, $\sigma^{r}$ acts as the additive inversion map on $\F$ (multiplication by $-1$), and so $\sigma^r$ fixes each hyperplane setwise.  Noting that the Frobenius automorphism $F: x \mapsto x^3$ is an automorphism of $\F$ of order $3d$, the automorphism $F^d: x \mapsto x^{3^d}$ is an automorphism of order $3$ of $\F$.  Since $3$ is coprime to $r$, $F^d$ must fix at least one hyperplane; indeed, without loss of generality we will assume that $F^d$ fixes the hyperplane
 \[ H_0 \colonequals \langle a_1, \dots, a_{3d - 1} \rangle.\]
 Thus, we have
 \[ H_i^{F^d} = \left(H_0 \cdot \theta^i \right)^{F^d} = H_0^{F^d} \cdot \left(\theta^i\right)^{F^d} = H_0 \left(\theta^{ 3^d}\right)^i = H_{i \cdot 3^d},\]
 where the calculation of the subscript $i\cdot 3^d$ is done modulo $r$.  (Indeed, $F^{d}$ is acting on $\F^*/\langle -1 \rangle \cong C_r$ exactly as $\rho$ is acting on $\langle w \rangle$.)
 
 By \cite[Theorem 1]{Spence_1977}, if we denote the $r$ elements of $\langle w \rangle$ by $k_0, \dots, k_{r-1}$ (in any order), then the set
 \[D \colonequals H_0^c k_0 + \sum_{j = 1}^{r - 1} H_j k_j \]
 is a difference set (here, $H_0^c$ denotes the set-theoretic complement).  Here, we will specifically choose $k_i \colonequals w^i$ for each $i$, i.e., we will consider
 \[ D = H_0^c + \sum_{j = 1}^{r - 1} H_j w^j\]
 
 By the above discussion, we may define an automorphism $\phi$ of $G$, where, if $g = hw^i$ for $h \in H$, then
 \[ g^\phi \colonequals h^{F^d}(w^i)^\rho.\]
 Thus, 
 \begin{align*} 
    D^\phi &= \left( H_0^c  + \sum_{j = 1}^{r - 1} H_j w^j\right)^\phi\\
%              &= \left(H_0^c\right)^\phi + \sum_{j = 1}^{r - 1} H_j^\phi \left(w^j\right)^\phi \\
%              &= \left(H_0^c\right)^{F^d} + \sum_{j = 1}^{r - 1} H_j^{F^d} \left(w^j\right)^\rho \\
             &= H_0^c + \sum_{j = 1}^{r - 1} H_{j \cdot 3^d} w^{j \cdot 3^d} \\
             &= D,\\
 \end{align*}
 and so $\phi \in \Aut(G)_{D}.$
 
 Define $\calG \le \Aut(G)_{D} \ltimes G$ by 
 \[ \calG \colonequals \langle (1,a_1), \dots, (1,a_{3d - 1}), (\phi, a_{3d}), (1,w) \rangle.\] 
 Note that, if 
 \[ X \colonequals \langle a_1, \dots, a_{3d - 1}, w \rangle = \langle H, w \rangle,\]
 then $\calG \cap (1 \times G) = 1 \times X$ and $|G:X| = 3$.  Moreover, 
 \[ (\phi,a_{3d})^{-1} (1 \times X) (\phi, a_{3d}) = 1 \times a_{3d}^{-1} X^{\phi} a_{3d} =  1 \times X,\]
 and so $1 \times X$ is a normal subgroup of both $1 \times G$ and $\calG$.
 
 Next, examining the action of $(\phi, a_{3d})$ on 
 \[[G:X] = \left\{X, Xa_{3d}, Xa_{3d}^2\right\},\] 
 we see that 
 \[X^{(\phi, a_{3d})} = X^\phi a_{3d} = Xa_{3d};\]
 since the orbit of $\calG$ containing $X$ has size $1$ or size $3$, we see that $\calG$ is in fact transitive on $[G:X]$ under this action.
 
 We note that $H_0^\phi = H_0$ and $a_{3d} \notin H_0$, so $a_{3d}^\phi = ha_{3d}^i$ for some $h \in H_0$ and $i = 1, 2$.  Moreover, $\phi$ has order $3$, so we must have $i = 1$, i.e., $a_{3d}^\phi = h a_{3d}$ for some $h \in H_0$.  With this in mind, we have
 \[ (\phi, a_{3d})^3 = (\phi^3, (a_{3d})^{\phi^2} (a_{3d})^\phi a_{3d}) = (1, h^\phi h a_{3d}^3) = (1, h^\phi h) \in 1 \times H_0 \le 1 \times X,\]
 and so the order of $(\phi, a_{3d})$ is $9$, and $|\calG:(1 \times X)| = 3$.  Therefore, $\calG$ satisfies the criteria of Theorem \ref{thm:transfer}, and so $\calG$ contains a DS with Spence parameters.  
 
 To finish the proof, we will show that $\calG$ contains a nonabelian, non-normal Sylow $3$-subgroup.  First, by order considerations, we see that 
 \[ P_3 \colonequals \langle H_0, (\phi, a_{3d}) \rangle \]
 is a Sylow $3$-subgroup of $\calG$.  To see that $P_3$ is nonabelian, we first take $h \in H_0$ such that $h^\phi = h^{F^d} = h^{3^d} \neq h$; such an $h$ exists because $F^d$ fixes exactly $3^d$ elements of $H$ and $|H| = 3^{3d - 1}.$  Then,
 \[ (\phi, a_{3d})(1, h) = (\phi, ha_{3d}),\]
 whereas
 \[ (1, h)(\phi, a_{3d}) = (\phi, h^\phi a_{3d}),\]
 and we see that $P_3$ is nonabelian.  Finally, we note that
 \[ (1, w)(\phi, a_{3d})(1,w)^{-1} = (\phi, w^\rho a_{3d} w^{-1}) = (\phi, a_{3d} w^{3^d - 1}) \notin P_3,\]
 and so $\calG$ indeed has a nonabelian, non-normal Sylow $3$-subgroup, as desired.
 \end{proof}

\begin{rem} 
When $d = 1$ in Theorem \ref{thm:spencenonab3}, we obtain the nonabelian Spence DS first noted in \cite[Example 4.1]{Davis_etal_2024}, as we see in the following example. 
\end{rem}

 \begin{ex}
  We consider the concrete example when $d = 1$ and $r = 13$.  In this instance, $\F = \GF(27) = \GF(3)(\theta)$, where $\theta^3 = \theta + 2$.  The map $\sigma: \F \to \F$ defined by $x^\sigma = x \cdot \theta$ is an additive automorphism.    If we consider $\F$ as a $3$-dimensional vector space over $\GF(3)$, then the hyperplanes are the $2$-dimensional subspaces of $\F$.  We define
  \[ H_0 \colonequals \langle 1, \theta \rangle = \{a + b \theta : a, b \in \GF(3) \}.\]
  With this ``base hyperplane,'' we then define the other twelve hyperplanes by 
  \[ H_i \colonequals H_0^{\sigma^i} = H_0 \cdot \theta^i = \{ h \cdot \theta^i : h \in H_0\}.\]
  
  We may further consider the Frobenius mapping $F: \F \to \F$ defined by $x^F = x^3$.  Since $x^{F^3} = x^{27} = x$ for all $x \in \F$, we have that $F^1 = F$ is an element of order $3$ in the automorphism group of $\F$.  We observe that
  \[H_0^F = \langle 1^F, \theta^F \rangle = \langle 1, \theta^3 \rangle = \langle 1, \theta + 2 \rangle = H_0,\]
  and the other twelve hyperplanes are permuted using the permutation \[\pi \colonequals (1 \; 3 \; 9)(2 \; 6 \; 5)(4 \; 12 \; 10)(7 \; 8 \; 11)\]
  on the indices, i.e., $H_i^F = H_{i^\pi}.$  Noting that $\theta^{13} = -1$, we can further think of the Frobenius map $F$ restricted to the quotient group $K \colonequals \F^*/\langle -1 \rangle = \langle \overline{\theta} \rangle \cong C_{13}$; here, $F$ will also act on $K$ via the permutation $\pi$ in the sense that $\left(\overline{\theta}^i\right)^F = \overline{\theta}^{i^\pi}.$
  
  Putting this all together, the Frobenius map $F$ will act on \[G \colonequals \F^+ \times \F^*/\langle -1 \rangle \cong H \times K\] in such a way that the set
  \[ D \colonequals (H_0^c, \overline{1}) \cup \bigcup_{i = 1}^{12} (H_i, \overline{\theta}^i)\]
  is fixed under the action of $F$.  Now, $D$ is a Spence DS that is invariant under the action of $\langle F \rangle \ltimes G$, and, if we define 
  \[ X \colonequals \langle (h,k): h \in H_0, k \in  K \rangle,\]
  consider 
  \[ \calG \colonequals \langle \id_\F \times X, (F, (\theta^2, \overline{1})) \rangle.\]
  Now, 
  \begin{align*} 
  (F, (\theta^2, \overline{1}))^3 &= (F^3, (\theta^2, \overline{1})^{F^2}(\theta^2, \overline{1})^F (\theta^2, \overline{1}))\\ 
   &= (\id_\F, ( (\theta^2 + 2\theta + 1) + (\theta^2 + \theta + 1) + \theta^2, \overline{1}) )\\
   &= (\id_\F, (2, \overline{1})) \in \id_\F \times X,
  \end{align*}
  so $(F, (\theta^2, \overline{1}))$ has order $9$, and, since $(F, (\theta^2, \overline{1}))$ normalizes $\id_\F \times X$, \[|\calG| = |G| = 27 \cdot 13 = 351.\]
  In fact, $\calG \cong C_{13} \rtimes (C_9 \rtimes C_3)$, which can be identified as SmallGroup(351,7) in $\GAP$ \cite{GAP4}.  Since \[[G:X] = \{X, X(\theta^2, \overline{1}), X(2\theta^2, \overline{1}) \},\]
  and
  \[ X^{(F, (\theta^2, \overline{1}))} = X(\theta^2, \overline{1}),\]
  we see that $\calG$ indeed acts transitively on the vertices of the directed Cayley graph $\Cay(G, D)$.  Thus, $\calG$ also contains a DS with Spence parameters. 
\end{ex}

\section{Nonabelian Denniston PDSs}
\label{sect:Denniston}

This section is dedicated to the construction of Denniston PDSs in nonabelian groups.

\subsection{A nonabelian $2$-group containing a Denniston PDS whenever the elementary abelian group of the same order does}
\label{subsect:evenDennisfromelab}

In \cite{Denniston_1969}, Denniston proved an existence result concerning maximal arcs in projective spaces, which can be rephrased in terms of PDSs as follows (see \cite[Theorems 8.1, 9.1, 9.3, Example 9.4]{Ma_1994b}):

\begin{thm}
 \label{thm:originalDenniston}
 Let $Q: \GF(2^m)^2 \to \GF(2^m)$ be a quadratic form such that $Q(x) = 0$ if and only if $x = 0$, and let $K$ be a subgroup of the additive group of $\GF(2^m)$ with $2^r$ elements, where $1 \le r < m$.  Then 
 \[D \colonequals \{ (c, ca, cb): c \in \GF(2^m)^\times, Q(a,b) \in K \}\]
 is a PDS in the additive group $G = (\GF(2^m)^3, +)$ with parameters
 \begin{align*}
   v       &= 2^{3m}\\
   k       &= (2^{m+r} - 2^m + 2^r)(2^m - 1),\\
   \lambda &= 2^m - 2^r + (2^{m+r} - 2^m + 2^r)(2^r - 2),\\
   \mu     &= (2^{m+r} - 2^m + 2^r)(2^r - 1).
 \end{align*}
\end{thm}

By \cite[Theorem 3.28]{Ball_2015}, if $x = (x_1, x_2)$, we may choose $Q(x) = x_1x_2$ or $Q(x) = x_1^2 + \alpha x_1x_2 + x_2^2$, where $\alpha \in \GF(2^m)$ satisfies $\tr(\alpha^{-1}) = 1$ ($\tr$ is the standard field trace as defined in Section \ref{sect:Denniston}).  Since we want $Q(x) = 0$ if and only if $x = 0$, our form must be equivalent to $Q(x) = x_1^2 + \alpha x_1x_2 + x_2^2$, and in this case the matrix 
\[ M \colonequals \begin{pmatrix}
                   1 & 0 & 0 \\
                   0 & 0 & 1 \\
                   0 & 1 & 0 \\
                  \end{pmatrix}
\]
is an automorphism of $G$ that fixes $D$ setwise since switching the second and third coordinates stabilizes $Q$.  

With this setup, we can prove the existence of a nonabelian group containing Denniston PDSs. 

\begin{thm}
 \label{thm:nonabelianDennis2} 
 Let $m \ge 2$ be a positive integer.  There exists a nonabelian group $\calG$ that contains a PDS with Denniston parameters
 \begin{align*}
   v       &= 2^{3m}\\
   k       &= (2^{m+r} - 2^m + 2^r)(2^m - 1),\\
   \lambda &= 2^m - 2^r + (2^{m+r} - 2^m + 2^r)(2^r - 2),\\
   \mu     &= (2^{m+r} - 2^m + 2^r)(2^r - 1).
 \end{align*}
 for each positive integer $r$, $1 \le r < m$.
\end{thm}

\begin{proof}
 Let $G$, $D$, and $M$ be as defined above.  Define $u \colonequals (0,1,1) \in G$, and let $X$ be the complement hyperplane to $\langle u \rangle$ in $G$ (as a $\GF(2)$-vector space), such that
 \[ G = X \oplus \langle u \rangle.\]
 Since $M$ stabilizes $u$, $M$ stabilizes $X$ setwise.  Let $x_1, \dots x_{3m-1}$ be a $\GF(2)$-basis for $X$, and define
 \[ \calG \colonequals \langle (1,x_i), (M,u) : 1 \le i \le 3m - 1 \rangle.\]
 Note that $\calG \cap (1 \times G) = 1 \times X$, and, since $M$ stabilizes $X$,
 \[ (M, u)^{-1} (1 \times X) (M,u) = 1 \times X,\]
 meaning $1 \times X$ is a normal subgroup of $\calG$.  Furthermore, $G = X \cup (X + u)$, and 
 \[X^{(M,u)} = X^M + u = X + u,\]
 meaning $\calG$ is transitive on the cosets of $X$ in $G$.  Therefore, by Theorem \ref{thm:transfer}, $\calG$ contains a PDS with the same parameters as $D$.
 
 Finally, we note that
 \[ (1, (0,1,0)) \cdot (M, u) = (M, (0,1,0)^M + u) = (M, (0,0,1) + (0,1,1)) = (M, (0,1,0)),\]
 whereas
 \[ (M,u) \cdot (1, (0,1,0)) = (M, u + (0,1,0)) = (M, (0,0,1)),\]
 and so $\calG$ is indeed nonabelian.
\end{proof}

\subsection{Nonabelian groups containing a Denniston PDS related to non-elementary abelian $2$-groups}
\label{subsect:evenDennisnonelab}

We follow the notation of \cite{Davis_Xiang_2000} (and refer the interested reader there for further details not explicitly mentioned here).

We first construct a Galois ring over $\Z/4\Z$ as follows.  Let $\Phi_2(x) \in GF(2)[x]$ be a primitive polynomial of degree $t$.  There exists a unique polynomial $\Phi(x) \in \Z/4\Z[x]$ of degree $t$ such that $\Phi(x) \equiv \Phi_2(x) \pmod 2$ and $\Phi(x)$ divides $x^{2^t - 1} - 1 \in \Z/4\Z[x]$.  We denote the quotient ring $\Z/4\Z[x]/(\Phi(x))$ by $\GR(4,t)$; this is the \textit{Galois ring over $\Z/4\Z$ of degree $t$}.  We define $R \colonequals \GR(4,t)$ and $h$ to be a root of $\Phi(x)$ in $R$; note that $h$ has multiplicative order $2^t - 1$ and $R = \Z/4\Z[h]$. 

The ring $R$ is a finite local ring with unique maximal ideal $2R$, and $R/2R \cong \GF(2^t)$.  If we denote the natural projection from $R$ to $\GF(2^t)$ by $\pi$, then $g \colonequals \pi(h)$ is a primitive element of $\GF(2^t)$.  Moreover, as $\GF(2)$-vector spaces, $2R = \{0 ,2, 2h, \dots, 2h^{2^t - 2} \}$ and $\GF(2^t)$ are isomorphic, with an explicit isomorphism given by $\phi(2h^i) = g^i$ for $0 \le i \le 2^t - 2$ and $\phi(0) = 0$.

If $\tr: \GF(2^t) \to \GF(2)$ denotes the standard trace map
\[ \tr(\alpha) = \sum_{i = 0}^{t - 1} \alpha^{2^i},\]
then $H_0 \colonequals \{ r \in \GF(2^t) : \tr(r) = 0 \}$ is a hyperplane, and the multiplicative action of $g$ acts as a Singer cycle and is transitive on hyperplanes, i.e., the hyperplanes of $\GF(2^t)$ are $H_j \colonequals g^jH_0$, where $0 \le j \le 2^t - 2$.  Thus, the $(t-1)$-dimensional $\GF(2)$ subspaces of $2R$ are of the form $K_j \colonequals  \phi^{-1}(H_j)$, $0 \le j \le 2^t - 2$.  More explicitly, $K_0 = \{ r \in 2R : \tr(\phi(r)) = 0\}$ and $K_j = h^jK_0$ for $1 \le j \le 2^t - 2$.

We are now ready to define our PDS.  Define
\[ G \colonequals (R, +) \times (\GF(2^t), +) \cong C_4^t \times C_2^t.\]
Let $w \in R/2R$ with $\tr(w) = 1$; let $a_{i,j}$ be any element of $R$ such that
 \[ \pi(a_{i,j}) = g^i(1 + w) + g^jw,\]
 where $0 \le i,j \le 2^t - 2$;  and let
 \[ E_i \colonequals \bigcup_{j = 0}^{2^t - 2} (h^i + h^{2i -j} + 2a_{i,j} + K_j).\]
Define 
\[ D \colonequals \bigcup_{i = 0}^{2^t - 2} (E_i, g_i),\]
where $(E_i, g^i) = \{(r,g^i): r \in E_i\}.$   

\begin{thm}\cite[Theorems 3.2]{Davis_Xiang_2000}
 \label{thm:abelianDenniston}
 The set $D$ defined above is a $(2^{3t}, (2^{2t-1} - 2^{t-1})(2^t - 1), 2^{t-1} + (2^{2t-1} - 2^{t-1})(2^{t-1} - 2), (2^{2t-1} - 2^{t-1})(2^{t-1} - 1))$-PDS in $G \cong C_4^t \times C_2^t$.
\end{thm}

A PDS with the parameters listed in Theorem \ref{thm:abelianDenniston} is said to have \textit{Denniston parameters}.

We will explicitly define automorphisms of $G$ fixing $D$ when $t \ge 3$.  First, since $G$ is an abelian group and $D$ is a \textit{Type II PDS}, Ma's Multiplier Theorem \cite{Ma_1984} implies that the inversion map $\psi_0: G \to G$ defined by
\[\psi_0: (r, g^i) \mapsto (-r, -g^i) = (r + 2r, g^i)\] is an automorphism of $G$ fixing $D$ setwise.    

We define further automorphisms $\psi_\ell: G \to G$, $1 \le \ell \le t-1$, as follows.  First, define $f_\ell: \GF(2^t) \to 2R$ by $f_\ell(0) \colonequals 0$ and
\[ f_\ell(g^i) \colonequals \phi^{-1}(\tr(g)) \cdot h^i + \sum_{\substack{0 \le k \le t-1,\\ k \neq \ell-1, \ell-2 \mod t}}\limits 2h^{i + 2^k}.\] 
Now, we define $\psi_\ell: G \to G$ by
\[ \psi_\ell: (r, s) \mapsto (r + 2rh^{2^{\ell-1}} + 2f_\ell(s), s).\]
Indeed, $\psi_\ell$ is an automorphism since generating elements are sent to generating elements.  Furthermore, note that for any $2r \in 2R$, $\psi_\ell((2r,0)) = (2r,0)$, so $\psi_\ell$ fixes the elements of $(2R,0)$.  Furthermore, like the inversion map $\psi_0$, it is clear that each $\psi_\ell$ has order $2$, and explicit calculation shows that $\psi_i$ and $\psi_j$ commute for all $i,j$, and collectively the $\psi_\ell$, $0 \le \ell \le t-1$, generate an elementary abelian group of automorphisms of order $2^t$.

\begin{lem}
  \label{lem:autDennis}
  The automorphism $\psi_\ell$, $1 \le \ell \le t-2$, fixes $D$ setwise.
\end{lem}

\begin{proof}
 It suffices to prove that $\psi_\ell$ fixes each $(E_i, g_i)$ setwise, and hence it suffices to prove that $\psi_\ell$ fixes $(h^i + h^{2i -j} + 2a_{i,j} + K_j, g^i)$ setwise for each $i$ and $j$.  Moreover, $\psi_\ell$ fixes the element $2a_{i,j}$ and the set $K_j$, so it suffices to prove that $(h^i + h^{2i - j}, g^i)^{\psi_\ell} - (h^i + h^{2i - j}, g^i) \in K_j = h^jK_0$ for each $i,j$.  Note that 
 \begin{align*} 
 (h^i + h^{2i - j}, g^i)^{\psi_\ell} &- (h^i + h^{2i - j}, g^i) = (2h^{i+2^{\ell -1}} + 2h^{2i - j + 2^{\ell - 1}} + 2f_\ell(g^i), 0)\\ 
      &= \left(2h^{i+ 2^{\ell - 1}} + 2h^{2i - j + 2^{\ell - 1}} + \phi^{-1}(\tr(g))\cdot h^i + \sum_{\substack{0 \le k \le t-1,\\ k \neq \ell-1, \ell-2 \mod t}}\limits 2h^{i + 2^k}, 0 \right). 
 \end{align*}
  This element will be in $K_j$ if and only if 
 \[ 2h^{(i - j) + 2^{\ell - 1}} + 2h^{2(i - j) + 2^{\ell - 1}} + \phi^{-1}(\tr(g))\cdot h^{i-j} + \sum_{\substack{0 \le k \le t-1,\\ k \neq \ell-1, \ell-2 \mod t}}\limits 2h^{(i - j) + 2^k} \in K_0,\]
and this holds if and only if, for each $m$,
 \[ \tr\left(g^{m+ 2^{\ell - 1}} + g^{2m+ 2^{\ell - 1}} + \tr(g)\cdot g^m + \sum_{\substack{0 \le k \le t-1,\\ k \neq \ell-1, \ell-2 \mod t}}\limits g^{m + 2^k}\right) = 0.\]
Now, $\tr$ is additive and, for all $\alpha \in \GF(2^t)$, $\tr(\alpha) = \tr(\alpha^2)$, so, indeed,
\begin{align*}
 &\tr\left(g^{m+ 2^{\ell - 1}} + g^{2m+ 2^{\ell - 1}} +  \tr(g)\cdot g^m + \sum_{\substack{0 \le k \le t-1,\\ k \neq \ell-1, \ell-2 \mod t}}\limits g^{m + 2^k}\right)\\ 
 &\text{       }= \tr\left( g^{2m+2^\ell} + g^{2m+ 2^{\ell - 1}} + \tr(g)\cdot g^{2m} + \sum_{\substack{0 \le k \le t-1,\\ k \neq \ell-1, \ell-2 \mod t}}\limits g^{2m + 2^{k+1}}\right)\\
 %&\text{       }= \tr\left(\tr(g) \cdot g^{2m} + \sum_{k = 0}^{t - 1} g^{2m + 2^{k+1}}\right)\\
 %&\text{       }= \tr\left(\tr(g) \cdot g^{2m} + g^{2m} \cdot \sum_{k=0}^{t-1} g^{2^{k+1}}\right)\\
  %&\text{       }= \tr\left(\tr(g) \cdot g^{2m} + g^{2m} \cdot \sum_{j = 0}^{t - 1} g^{2^j}\right)\\
  &\text{       }= \tr\left(\tr(g) \cdot g^{2m} + g^{2m} \cdot \tr(g) \right)\\
  %&\text{       }= \tr(0)\\
  &\text{       }= 0,
\end{align*}
and, therefore, the automorphism $\psi_\ell$ fixes $D$ setwise.
 \end{proof}

\begin{thm}
 \label{thm:nonabelianDennis}
 Let $t \ge 2$ be a positive integer.  For each $k$, $1 \le k \le t$, here exists a nonabelian group \[\calG \cong (C_4^t \times C_2^{t-k}) \rtimes (C_2^k)\] that contains a PDS with Denniston parameters; namely, $\calG$ contains a PDS with parameters
 \[ (2^{3t}, (2^{2t-1} - 2^{t-1})(2^t - 1), 2^{t-1} + (2^{2t-1} - 2^{t-1})(2^{t-1} - 2), (2^{2t-1} - 2^{t-1})(2^{t-1} - 1)).\]
\end{thm}

\begin{proof}
 Let $G$ and $D$ be as defined above, and let $k$ be a fixed integer, $1 \le k \le t$.  Define $\calG \le \Aut(G)_{D} \ltimes G$ by
 \[ \calG = \langle (1, (h^i,0)), (1, (0, g^j)), (\psi_\ell, (0,g^\ell)) : 0 \le i \le t-1, k \le j \le t-1, 0 \le \ell \le k-1 \rangle.\]
 
 Note that, if
 \[ X \colonequals \langle (h^i,0), (0, g^j) : 0 \le i \le t-1, k \le j \le t-1 \rangle, \]
 then $\calG \cap (1 \times G) = 1 \times X$ and $|G:X| = 2^k$.  Moreover, for each $\ell$, $0 \le \ell \le k-1$, we have
 \[ (\psi_\ell, (0, g^{\ell}))^{-1} (1 \times X)(\psi_\ell, (0, g^{\ell})) = 1 \times ((0,2g^{\ell}) + X^\psi_\ell) = 1 \times X,\]
 and so $1 \times X$ is a normal subgroup of both $1 \times G$ and $\calG$.  Finally, the cosets of $X$ in $G$ are of the form $\left(0, \sum_{i \in I}\limits g^i\right) + X$, where $I$ is a (possibly empty) subset of $\{0, \dots, k-1\}$, and, when $I$ is nonempty, we have
 \[ X^{\left( \prod_{i \in I} \psi_i, \sum_{i \in I} g^i\right)}= \left(0, \sum_{i \in I} g^i\right) + X^{\prod_{i \in I} \psi_i} = \left(0, \sum_{i \in I} g^i\right) + X,\]
so $\calG$ is indeed transitive on the cosets of $X$ in $G$.  Therefore, $\calG$ -- which, by construction is nonabelian and isomorphic to $(C_4^t \times C_2^{t-k}) \rtimes (C_2^k)$ -- satisfies the criteria of Theorem \ref{thm:transfer}, and so $G$ contains a PDS with Denniston parameters.
\end{proof}

\begin{ex}
 We consider the concrete example when $t = 3$.  Here, we choose the polynomial $\Phi(x) = x^3 + 2x^2 + x - 1$, so $\Phi_2(x) = x^3 + x + 1$.
 Note that \[\tr(g) = g + g^2 + g^4 = g \cdot \Phi_2(g) = 0,\]
 so $K_0 \langle 2h, 2h^2 \rangle$.
  
 To make the notation somewhat easier, we will consider the group
 \[ G = \langle a, b, c, d, e, f \rangle,\]
 where $G$ is abelian and 
 \[ a^4 = b^4 = c^4 = d^2 = e^2 = f^2 = 1.\]
 In this case, $a$ corresponds to $1 \in R$ (and $a^2$ corresponds to $2$), $b$ corresponds to $h \in R$ (and $b^2$ corresponds to $2h$), $c$ corresponds to $h^2 \in R$ (and $c^2$ corresponds to $2h^2$), $d$ corresponds to $1 \in \GF(8)$, $e$ corresponds to $g \in \GF(8)$, and $f$ corresponds to $g^2 \in \GF(8)$.  Thus, our $K_j$ will be:
 \begin{align*}
  K_0 &= \langle b^2, c^2 \rangle,\\
  K_1 &= \langle c^2, a^2b^2 \rangle,\\
  K_2 &= \langle a^2b^2, b^2c^2 \rangle,\\
  K_3 &= \langle b^2c^2, a^2b^2c^2 \rangle,\\
  K_4 &= \langle a^2b^2c^2, a^2c^2 \rangle,\\
  K_5 &= \langle a^2c^2, a^2 \rangle,\\
  K_6 &= \langle a^2, b^2 \rangle;
 \end{align*}
using group ring notation, we may choose the $E_i$ to be:
\begin{align*}
 E_0 &= K_0 + a^2c K_1 + b^3c^3 K_2 + abc^3 K_3 + a^2b K_4 + a^3b^2c^3 K_5 + a^3bc^2 K_6, \\
E_1 &= a^2bc K_0 + K_1 + abc^2 K_2 + a^3bc K_3 + a^3b^2c^3 K_4 + b^2c K_5 + a^3 K_6,\\
E_2 &= b^3 K_0 + abc^3 K_1 + K_2 + a^2b^3c K_3 + ab^2c^3 K_4 + a^3 K_5 + ab^3 K_6,\\
E_3 &= bc^3 K_0 + c^3 K_1 + a^3b^2c^3 K_2 + K_3 + abc K_4 + a^3b^2 K_5 + b^3 K_6,\\
E_4 &= c^3 K_0 + a^3bc^3 K_1 + a^3bc^2 K_2 + a^3 K_3 + K_4 + ac^3 K_5 + b^3c^2 K_6,\\
E_5 &= a^2b^2c^3 K_0 + a^3bc^2 K_1 + a^3c^3 K_2 + a^2bc K_3 + b^3 K_4 + K_5 + a^3b^2c^2 K_6,\\
E_6 &= a^2bc^2 K_0 + a^3b^3 K_1 + a^2bc K_2 + a^3c^2 K_3 + abc^3 K_4 + c^3 K_5 + K_6,\\
\end{align*}
and we have
\[ D = d E_0 + e E_1 + f E_2 + de E_3 + ef E_4 + def E_5 + df E_6.\]

We define $\psi_0$ to be the inversion map on $G$, and, since $\tr(g) = 0$, we define $\psi_1$ and $\psi_2$ as follows:

\begin{align*}
 a^{\psi_1} &= ab^2 \leftrightarrow (1,0)^{\psi_1} = (1 + 2h, 0),\\
 b^{\psi_1} &= bc^2 \leftrightarrow (h, 0)^{\psi_1} = (h + 2h^2, 0),\\
 c^{\psi_1} &= a^2b^2c \leftrightarrow (h^2,0)^{\psi_1} = (h^2 + 2h^3,0) = (h + (2 + 2h), 0),\\
 d^{\psi_1} &= c^2d \leftrightarrow (0,1)^{\psi_1} = \left(\sum_{\substack{0 \le k \le 2,\\ k \neq 0, 2 \mod 3}}2h^{0 + 2^k}, 1 \right) = (2h^2, 1)\\
 e^{\psi_1} &= a^2b^2e \leftrightarrow (0,g)^{\psi_1} = \left(\sum_{\substack{0 \le k \le 2,\\ k \neq 0, 2 \mod 3}}2h^{1 + 2^k}, g \right) = (2h^3, g) = (2 + 2h, g),\\
 f^{\psi_1} &= b^2c^2f \leftrightarrow (0,g^2)^{\psi_1} = \left(\sum_{\substack{0 \le k \le 2,\\ k \neq 0, 2 \mod 3}}2h^{2 + 2^k}, g^2 \right) = (2h^4,g^2) = (2h + 2h^2, g^2);
\end{align*}

\begin{align*}
 a^{\psi_2} &= ac^2 \leftrightarrow (1,0)^{\psi_2} = (1 + 2h^2, 0),\\
 b^{\psi_2} &= a^2b^3 \leftrightarrow (h, 0)^{\psi_2} = (h + 2h^3, 0) = (h + (2 + 2h), 0),\\
 c^{\psi_2} &= b^2c^3 \leftrightarrow (h^2,0)^{\psi_2} = (h^2 + 2h^4,0) = (h^2 + (2h + 2h^2), 0),\\
 d^{\psi_2} &= b^2c^2d \leftrightarrow (0,1)^{\psi_2} = \left(\sum_{\substack{0 \le k \le 2,\\ k \neq 1, 0 \mod 3}}2h^{0 + 2^k}, 1 \right) = (2h^4, 1) = (2h + 2h^2, 1)\\
 e^{\psi_2} &= a^2b^2c^2e \leftrightarrow (0,g)^{\psi_2} = \left(\sum_{\substack{0 \le k \le 2,\\ k \neq 1, 0 \mod 3}}2h^{1 + 2^k}, g \right) = (2h^5, g) = (2 + 2h + 2h^2, g),\\
 f^{\psi_2} &= a^2c^2f \leftrightarrow (0,g^2)^{\psi_2} = \left(\sum_{\substack{0 \le k \le 2,\\ k \neq 1, 0 \mod 3}}2h^{2 + 2^k}, g^2 \right) = (2h^6,g^2) = (2 + 2h^2, g^2).
\end{align*}

Indeed, $\psi_0$, $\psi_1$, and $\psi_2$ are each automorphisms of $G$ fixing $D$ setwise, and, if we define
\[ \calG \colonequals \langle (1,a), (1,b), (1,c), (\psi_0, d), (\psi_1, e), (\psi_2, f) \rangle,\]
then $\calG$ has order $512$, and, if $X = \langle a, b, c\rangle$, then $\calG \cap (1 \times G) = 1 \times X$.  It is easy to check that the action of $\calG$ on the cosets of $X$ in $G$ is transitive, and so Theorem \ref{thm:transfer} guarantees the existence of a PDS with Denniston parameters in $\calG \cong C_4^3 \rtimes C_2^3$.
 
\end{ex}

\subsection{Nonabelian Denniston PDSs in groups of odd order}

In \cite{Davis_etal_2023b}, the authors construct PDSs with Denniston parameters in a group of order $p^{3m}$, where $p$ is an odd prime.  (See also \cite{DeWinter_2023} and, more recently, \cite{Bao_etal_2024}.)

\begin{thm}\cite[Theorem 5.1]{Davis_etal_2023b}
 \label{thm:oddDenniston}
 Let $p$ be an odd prime and $m \ge 1$ be an integer.  If $\omega$ is a primitive element of $\GF(p^m)$ and $\alpha$ is a primitive element of $\GF(p^{2m})$, then the set
 \[ D \colonequals \bigcup_{i = 0}^{\frac{p^m - 1}{p - 1}-1} \left(\omega^i \left\langle \omega^{\frac{p^m - 1}{p - 1}} \right\rangle\right) \times \left(\alpha^i \left\langle \alpha^{\frac{p^m - 1}{p -1}} \right\rangle \cup \left\{0_{\GF(p^{2m}}) \right\}\right)\]
 is a PDS in the group $G \colonequals (\GF(p^m), +) \times (\GF(p^{2m}), +) \cong C_p^{3m}$ with parameters
 \begin{align*}
  v       &= p^{3m}, \\
  k       &= (p^m - 1)((p - 1)(p^m + 1) + 1),\\
  \lambda &= p^m - p + (p^{m+1} - p^m + p)(p - 2),\\
  \mu     &= (p^{m+1} - p^m + p)(p - 1).
 \end{align*}

\end{thm}

Assume henceforth that $m$ is divisible by $p$, i.e., that $m = pt$ for some integer $t$, and define an automorphism $\phi$ on $G$ by 
\[ (x,y)^\phi \colonequals (x^{p^{2t}}, y^{p^{2t}}).\]
To see that this is an automorphism of $G$, we note that, if $\sigma_1$ is the Frobenius automorphism on $\GF(p^m)$ and $\sigma_2$ is the Frobenius automorphism on $\GF(p^{2m})$, then $\phi = \sigma_1^{2t} \times \sigma_2^{2t}$.  Moreover, since $\sigma_2$ has order $2m = 2pt$, $\sigma_2^{2t}$ has order $p$, and, since $\sigma_1$ has order $m = pt$ and $p$ is coprime to $2$, $\sigma_1^{2t}$ also has order $p$; hence, $\phi$ is an automorphism of order $p$.

Furthermore, since $p$ is coprime to $(p^m - 1)/(p - 1)$, we have
\[ \left(\omega^i \left\langle \omega^{\frac{p^m - 1}{p - 1}} \right\rangle\right)^{\sigma_1^{2t}} = \left(\omega^{ip^{2t}} \left\langle \omega^{\frac{p^m - 1}{p - 1}} \right\rangle\right),\]
\[ \left(\alpha^i \left\langle \alpha^{\frac{p^m - 1}{p -1}} \right\rangle \cup \left\{0_{\GF(p^{2m}}) \right\}\right)^{\sigma_2^{2t}} = \left(\alpha^{ip^{2t}} \left\langle \alpha^{\frac{p^m - 1}{p -1}} \right\rangle \cup \left\{0_{\GF(p^{2m}}) \right\}\right),\]
and hence $\phi$ fixes $D$ setwise.

\begin{thm}
 \label{thm:nonabelianoddDennis}
 Let $p$ be an odd prime and $m = pt$ be a multiple of $p$.  There exists a nonabelian group $\calG$ that contains a PDS with Denniston parameters
 \begin{align*}
  v       &= p^{3m}, \\
  k       &= (p^m - 1)((p - 1)(p^m + 1) + 1),\\
  \lambda &= p^m - p + (p^{m+1} - p^m + p)(p - 2),\\
  \mu     &= (p^{m+1} - p^m + p)(p - 1).
 \end{align*}
\end{thm}

\begin{proof}
 The proof is similar to that of Theorem \ref{thm:nonabelianDennis2}. Let $G$, $D$, and $\phi$ be as defined above.  Define $X$ to be a complement hyperplane to $(0,1) \in G = (\GF(p^m), +) \times (\GF(p^{2m}), +)$, where $G$ is considered as a $\GF(p)$-vector space.  Thus, 
 \[ G = X \oplus \langle (0,1) \rangle.\]
 Since $\phi$ stabilizes $(0,1)$, $\phi$ stabilizes $X$ setwise.  Let $x_1, \dots x_{3m-1}$ be a $\GF(p)$-basis for $X$, and define
 \[ \calG \colonequals \langle (1,x_i), (\phi, (0,1)) : 1 \le i \le 3m - 1 \rangle.\]
 Note that $\calG \cap (1 \times G) = 1 \times X$, and, since $\phi$ stabilizes $X$,
 \[ (\phi, (0,1))^{-1} (1 \times X) (\phi,(0,1)) = 1 \times X,\]
 meaning $1 \times X$ is a normal subgroup of $\calG$.  Furthermore, $G = \bigcup_{i = 0}^{p - 1}\limits X + i(0,1)$, and 
 \[X^{(\phi^i, i(0,1))} = X^{\phi^i} + i(0,1) = X + i(0,1),\]
 meaning $\calG$ is transitive on the cosets of $X$ in $G$.  Therefore, by Theorem \ref{thm:transfer}, $\calG$ contains a PDS with the same parameters as $D$.
 
 Finally, we note that
 \[ (1, (0, \alpha)) \cdot (\phi, (0,1)) = (\phi, (0, \alpha)^\phi + (0,1)) = (\phi, (0, \alpha^{p^{2t}}+1),\]
 whereas
 \[ (\phi, (0,1)) \cdot (1, (0, \alpha)) = (\phi, (0, \alpha + 1)),\]
 which are not equal since $\alpha$ is a primitive element with multiplicative order $p^{2m} - 1$, and so $\calG$ is indeed nonabelian.
\end{proof}

\begin{rem}
 Define $q \colonequals p^{t}$.  Then, $p^{m} = q^p$, and so $\GF(p^{m})$ is a degree $p$ extension of $\GF(q)$.  Define $\tr_{q}$ to be the trace map of $\GF(p^{m})$ over $\GF(q)$, i.e.,
 \[ \tr_q(\beta) = \sum_{i = 0}^{p-1} \beta^{q^i};\]
 similarly, $p^{2m} = (q^2)^p$, $\GF(p^{2m})$ is a degree $p$ extension of $\GF(q^2)$, and we define $\tr_{q^2}$ to be the trace map of $\GF(p^{2m})$ over $\GF(q^2)$.
 
 Note first that, for any $(\phi^i, (\beta, \gamma)) \in \calG$, since $p$ is odd, direct calculation shows that 
 \[ (\phi^i, (\beta, \gamma))^p = ((\phi^i)^p, (\tr_q(\beta), \tr_{q^2}(\gamma))) = (1, (\tr_q(\beta), \tr_{q^2}(\gamma))) \in 1 \times G,\]
 so every element of $\calG$ has order at most $p^2$.  On the other hand, let $\theta \in \GF(p^m)$ be any element with $\tr_q(\theta) \neq 0$.  Then, $(\phi, (\theta, 1)) = ((\phi, (0,1))(1, (\theta, 0)) \in \calG$, and 
 \[ (\phi, (\theta, 1))^p = (1, (\tr_q(\theta), \tr_{q^2}(1))) = (1, (\tr_q(\theta), 0)) \neq (1, (0,0)) = 1_\calG, \]
 so $\calG$ has exponent $p^2$.
\end{rem}

\section{Nonabelian DSs with McFarland parameters}
\label{sect:McFarland}

In \cite{McFarland_1973}, McFarland presents a very general construction for difference sets.  Let $q$ be a prime power and let $V$ be an $(s+1)$-dimensional $\GF(q)$-vector space.  Define 
\[ r \colonequals \frac{q^{s+1} - 1}{q - 1},\]
and let $H_1, \dots, H_r$ be the distinct hyperplanes of $V$.  Let $E = (V, +)$, the additive group of $V$, and consider the $H_i$ as subgroups of $E$.  Let $K$ be any group of order $r+1$, $k_1, \dots, k_r$ any $r$ distinct elements of $K$, and $e_1, \dots, e_r$ any $r$ elements of $E$ (not necessarily distinct).

\begin{thm}\cite{McFarland_1973}
\label{thm:McFarland}
 Using the notation above, if $G = E \times K$, then the set
 \[ D \colonequals \bigcup_{i = 1}^r (H_i + e_i, k_i)\]
 is a difference set in $G$ with parameters
 \begin{align*}
  v       &= q^{s+1}\left(\frac{q^{s+1} - 1}{q - 1} + 1\right),\\
  k       &= q^s \left( \frac{q^{s+1} - 1}{q - 1}\right),\\
  \lambda &= q^s\left(\frac{q^s - 1}{q - 1} \right).
 \end{align*}
\end{thm}

These results were generalized by Dillon, who was able to significantly weaken the necessary conditions of Theorem \ref{thm:McFarland}.

\begin{thm}\cite[Section 2, Theorem]{Dillon_1985}
 \label{thm:DillonMcFarland}
 Let $q$ be a prime power and let $s$ be any positive integer.  Let $E$ be an elementary abelian group of order $q^{s+1}$ (which we regard as vector space of dimension $s+1$ over $\GF(q)$), and let $H_1, \dots, H_r$, $r = (q^{s+1} - 1)/(q - 1)$, be the $s$-dimensional subspaces of $E$.  Let $G$ be a group of order $q^{s+1}(r+1)$ which contains $E$ as a normal subgroup and let $g_1, \dots, g_r$ be elements of $G$ lying in distinct cosets of $E$. Let 
 \[ D \colonequals \bigcup_{i = 1}^r g_i H_i.\]
 If the map 
 \[H_i \mapsto g_i H_i g_i^{-1} \]
 is a permutation of the hyperplanes of $E$, then $D$ is a difference set with McFarland parameters.
\end{thm}

Drisko \cite{Drisko_1998} was further able to provide conditions guaranteeing in certain instances the existence of the coset representatives $g_i$ that are necessary for Theorem \ref{thm:DillonMcFarland}:

\begin{cor}\cite[Corollary 8]{Drisko_1998}
 \label{cor:Drisko}
 Suppose $q$ is a prime, $r = (q^{s+1} - 1)/(q - 1)$, and $r + 1$ is a prime power.  Then any group $G$ of order $q^{s+1}(r+1)$ which has a normal elementary abelian subgroup $E$ of order $q^{s+1}$ has a difference set with McFarland parameters.
\end{cor}

\begin{rem}
It is likely that many McFarland DSs constructed using Theorem \ref{thm:transfer} from a group $G = E \times K$, where $E$ is elementary abelian, also can be shown to exist using the results of Dillon (Theorem \ref{thm:DillonMcFarland}) and Drisko (Corollary \ref{cor:Drisko}).  On the other hand, in the new groups created by Theorem \ref{thm:transfer}, it may not be obvious from the outset that the coset representatives necessary to apply Theorem \ref{thm:DillonMcFarland} exist, and it is also possible that the new groups created by Theorem \ref{thm:transfer} do not contain an elementary abelian subgroup of order $q^{s+1}$ and will be genuinely new.  
\end{rem}

With this discussion in mind, we present some constructions, some of which may also follow from Theorem \ref{thm:DillonMcFarland}.

\subsection{Constructions for even values of $q$}
\label{subsect:mcfareven}

In this subsection, we let $q = 2^d$ for some positive integer $d$.  We fix $s = 1$, $V = \GF(q)^2$, $r = q + 1$, $E = (V,+) \cong C_2^{2d}$, and $K = C_{r+1} = C_{q+2}$.  Then, 
\[G = E \times K \cong C_2^{2d} \times C_{q+2}.\]

Define an automorphism $\phi$ of $G$ as follows: if $E$ is identified with $V = \GF(q)^2$ and $K$ is viewed additively, then define
\[ \left((a,b), x\right)^\phi \colonequals \left((b,a), -x \right).\]
It is clear that $\phi$ is an order $2$ automorphism.  Note that $\langle (1,1) \rangle$ is the unique $\GF(q)$-hyperplane of $V$ fixed by swapping the coordinates.  Moreover, $(q + 2)/2 = q/2 + 1$ is odd, so $C_{q+2}$ contains a unique order $2$ element $u$ that will be fixed under the mapping $x \mapsto -x$.  

By Theorem \ref{thm:McFarland}, we may assign the nonidentity elements of $K$ to the hyperplanes $H_i$ of $E$ however we want, and we will obtain a difference set with McFarland parameters.  We choose the following assignment: first, if $H_1 = \langle (1,1) \rangle$, we choose $k_1 = u$; next, we order the other $H_i$ so that, if $H_i = \langle (a,b) \rangle$, then $H_{r+ 2 - i} = \langle (b,a) \rangle$; finally, we order the remaining $k_i$ so that, if $k_i = x$, then $k_{r + 2 - i} = -x$.  Define
\[ D \colonequals \bigcup_{i = 1}^r (H_i, k_i).\]
Then, by Theorem \ref{thm:McFarland}, $D$ is a DS with McFarland parameters, and our choice of $H_i$ and $k_i$ shows that $D$ is invariant under $\phi$, since $(H_1, k_1)^\phi = (H_1, k_1)$ and $(H_i, k_i)^\phi = (H_{r+2-i}, k_{r+2-i})$ for $2 \le i \le r$.  We will now use Theorem \ref{thm:transfer} to construct different nonabelian groups containing McFarland DSs.  

For ease of notation, we write elements in $\langle \phi \rangle \ltimes G$ in the form $(a, b, c)$, where $a \in \langle \phi \rangle$, $b \in E$, and $c \in K$.  We will also denote the identity element of $E = (V,+)$ simply by $0$.  (We will use similar notation when referring to elements of $G$, and we will use notation along these lines for the rest of Section \ref{sect:McFarland}.)

\begin{thm}
 \label{thm:evennonabMcF1}
 Let $G = E \times K$, $D$, and $\phi$ be as above.  Let $e_1, \dots e_{2d}$ be generators for $E$, and let $K = \langle u, y \rangle$, so that $|y| = (q+2)/2$.  Define
 \[ \calG \colonequals \langle (\phi, 0, u), (1, 0, y), (1, e_i, 0): 1 \le i \le 2d \rangle.\]
 Then, $\calG$ is a nonabelian group containing a DS with McFarland parameters.
\end{thm}

\begin{proof}
 Note that, if \[X = \langle (1, 0 , y), (1, e_i, 0): 1 \le i \le 2d \rangle,\] then $1 \times X = (1 \times G) \cap \calG$ and $X^\phi = X$.  Moreover, $(\phi,0 ,u)^2 = (1,0,0)$ and
 \[ (\phi,0,u)^{-1}(1 \times X)(\phi,0,u) = 1 \times X\]
 so $|\calG| = |G|$ and $1 \times X$ is a normal subgroup of $\calG$.  Finally, $G = X \cup X + (0,u)$, and
 \[ X^{(\phi,0,u)} = X^\phi + (0,u) = X + (0,u),\] so $\calG$ is transitive on the cosets of $X$ in $G$.  Thus, by Theorem \ref{thm:transfer}, $\calG$ contains a DS with McFarland parameters.  Finally, if $e_1 = (1,0)$ and $e_2 = (0,1)$, then
 \[ (\phi, 0, u) \cdot (1, e_1, 0) = (\phi, e_1, u),\]
 whereas
 \[ (1, e_1, 0) \cdot (\phi, 0, u) = (\phi, e_2, u),\]
 so $\calG$ is nonabelian.
\end{proof}

\begin{rem}
 If $\calG$ is the nonabelian group constructed in Theorem \ref{thm:evennonabMcF1} and $K$ is the dihedral group of order $r+2$, then $\calG \cong E \rtimes K$, and $\calG$ can be seen to have a McFarland DS by applying Theorem \ref{thm:DillonMcFarland}.
\end{rem}

\begin{thm}
 \label{thm:evennonabMcF2}
 Let $G = E \times K$, $D$, and $\phi$ be as above.  Let $e_1, \dots, e_{2d}$ be generators for $E$, where $e_1 = (1,0)$ and $e_2 = (0,1)$, and let $K = \langle u, y \rangle$, so that $|y| = (q+2)/2$.  Define
 \[ \calG \colonequals \langle (\phi, e_1,0), (1, 0, y), (1, 0, u), (1, e_1 + e_2 ,0), (1, e_i, 0): 3 \le i \le 2d \rangle.\]
 Then, $\calG$ is a nonabelian group containing a DS with McFarland parameters.
\end{thm}

\begin{proof}
 Note that, if \[X = \langle (1, 0, y), (1, 0, u), (1, e_1 + e_2 ,0), (1, e_i, 0): 3 \le i \le 2d \rangle,\] then $1 \times X = (1 \times G) \cap \calG$ and $X^\phi = X$.  Moreover, \[(\phi,e_1,0)^2 = (\phi^2, e_2 + e_1, 0) = (1, e_1 + e_2 ,0) \in X\] and
 \[ (\phi, e_1 ,0)^{-1}(1 \times X)(\phi, e_1, 0) = 1 \times X\]
 so $|\calG| = |G|$ and $1 \times X$ is a normal subgroup of $\calG$.  Furthermore, $G = X \cup X + (e_1,0)$, and
 \[ X^{(\phi, e_1, 0)} = X^\phi + (e_1,0) = X + (e_1,0),\] so $\calG$ is transitive on the cosets of $X$ in $G$.  Thus, by Theorem \ref{thm:transfer}, $\calG$ contains a DS with McFarland parameters.  Finally, 
 \[ (\phi, e_1,0) \cdot (1, 0 ,y) = (\phi, e_1 ,y),\]
 whereas
 \[ (1, 0,y) \cdot (\phi, e_1 ,0) = (\phi, e_1 ,-y),\]
 so $\calG$ is nonabelian.
\end{proof}

\begin{rem}
 The group $\calG$ does contain a normal, elementary abelian group
 \[ \calE = \langle (1, 0, u), (1, e_1 + e_2,0), (1, e_i, 0): 3 \le i \le 2d\rangle \]
 of order $q^2$.  It is possible that the coset representatives of $\calE$ in $\calG$ satisfy the hypotheses of Theorem \ref{thm:DillonMcFarland}, but this is not immediately clear to the authors.
\end{rem}

Finally, we modify the above constructions to create a group that does not have a normal, elementary abelian subgroup of order $q^2$.  As above, we let $V = \GF(q)^2$, where $q = 2^d$, and $E = (V,+)$.  Define $K' = \langle y, u \rangle \colonequals D_{(q+2)/2}$, the dihedral group of order $q+2$, where $y^{\frac{q+2}{2}} = u^2 = u^{-1}xux = 1$.  Now, we define
\[ G' \colonequals E \times K'.\]

Define an automorphism $\psi$ of $G'$ by 
\[ ((a,b), x)^\psi = ((b,a), u^{-1}xu).\]  Indeed, $\psi$ is an order $2$ automorphism of $G'$.  As in the previous cases, $\langle (1,1) \rangle$ is the unique $\GF(q)$-hyperplane of $V$ fixed by swapping the coordinates, and, since $(q+2)/2$ is odd, the only nonidentity element of $K'$ fixed under conjugation by $u$ is $u$ itself.

Proceeding as before, we may assign the nonidentity elements of $K$ to the hyperplanes $H_i$ of $E$ however we want, and we will obtain a difference set with McFarland parameters.  We choose the following assignment: first, if $H_1 = \langle (1,1) \rangle$, we choose $k_1 = u$; next, we order the other $H_i$ so that, if $H_i = \langle (a,b) \rangle$, then $H_{r+ 2 - i} = \langle (b,a) \rangle$; finally, we order the remaining $k_i$ so that, if $k_i = x$, then $k_{r + 2 - i} = u^{-1}xu$.  Define
\[ D' \colonequals \bigcup_{i = 1}^r (H_i, k_i).\]
Then, by Theorem \ref{thm:McFarland}, $D'$ is a DS with McFarland parameters, and our choice of $H_i$ and $k_i$ shows that $D'$ is invariant under $\psi$, since $(H_1, k_1)^\psi = (H_1, k_1)$ and $(H_i, k_i)^\psi = (H_{r+2-i}, k_{r+2-i})$ for $2 \le i \le r$.  Note that we will again use coordinate notation similar to that used in Theorems \ref{thm:evennonabMcF1} and \ref{thm:evennonabMcF2}.

\begin{thm}
 \label{thm:evennonabMcF3}
 Let $G' = E \times K'$, $D'$, $y$, $u$, and $\psi$ be as above.  Let $e_1, \dots, e_{2d}$ be generators for $E$, where $e_1 = (1,0)$ and $e_2 = (0,1)$.  Define
 \[ \calG' \colonequals \langle (\psi, e_1,1), (1, 0, y), (1, 0, u), (1, e_1 + e_2, 1), (1, e_i, 1): 3 \le i \le 2d \rangle.\]
 Then, $\calG'$ is a nonabelian group containing a DS with McFarland parameters that does not have a normal Sylow $2$-subgroup and that does not contain a normal elementary abelian subgroup of order $q^2$.
\end{thm}

\begin{proof}
 Note that, if \[X = \langle (1, 0, y), (1, 0, u), (1, e_1 + e_2 ,1), (1, e_i, 1): 3 \le i \le 2d \rangle,\] then $1 \times X = (1 \times G') \cap \calG'$ and $X^\psi = X$.  Moreover, \[(\psi, e_1,1)^2 = (\psi^2, e_2 + e_1, 1) = (1, e_1 + e_2, 1) \in X\] and
 \[ (\psi, e_1,1)^{-1}(1 \times X)(\psi, e_1 ,1) = 1 \times X\]
 so $|\calG'| = |G'|$ and $1 \times X$ is a normal subgroup of $\calG'$.  Furthermore, $G' = X \cup X(e_1,1)$, and
 \[ X^{(\psi, e_1 ,1)} = X^\psi(e_1,1) = X(e_1,1),\] so $\calG$ is transitive on the cosets of $X$ in $G$.  Thus, by Theorem \ref{thm:transfer}, $\calG$ contains a DS with McFarland parameters.  Finally, since a Sylow $2$-subgroup of $\calG'$ has order $2q^2$ and $(\psi, e_1 ,1)$ has order $4$, we see that
 \[ \calQ' = \langle (1, 0, u), (\psi, e_1 ,1), (1, e_i, 1): 3 \le i \le 2d\rangle\]
 is a Sylow $2$-subgroup of $\calG'$, and hence 
 \[ \calE' = \langle (1, 0, u), (1, e_1 + e_2,1), (1, e_i, 1): 3 \le i \le 2d\rangle\]
 is the unique elementary abelian subgroup of order $q^2$ in $\calG'$.  However, 
 \[ (1, 0, y)^{-1}(1, 0, u)(1, 0, y) = (1, 0, uy^2) \notin \calQ', \]
 so neither $\calQ'$ nor $\calE'$ is normal in $\calG'$, completing the proof.
\end{proof}

\subsection{A construction for odd values of $q$}

We present another construction, which is similar in spirit to Drisko's (Corollary \ref{cor:Drisko}).

\begin{thm}
\label{thm:oddMcF}
 Let $q$ be an odd prime, $2 \le s < q$.  Define 
 \[ r \colonequals \frac{q^{s+1} - 1}{q - 1}\]
 and assume $r + 1 = 2p$, where $p$ is an odd prime.  Then, there exists a group $\calG$ with a non-normal, nonabelian Sylow $q$-subgroup containing a DS with McFarland parameters.
\end{thm}

\begin{proof}
 First, let $V = \GF(q)^{s+1}$, $E = (V,+)$, $K = (\Z/(r+1)\Z,+) \cong C_{r+1} = C_{2p}$, and $G = E \times K$.  Note that $\Aut(K) \cong C_{p-1}$, and $\Aut(K)$ has exactly four orbits on $K$: $\{0\}$, $\{p\}$, and two orbits of length $p - 1$.  Now, since $2p = r + 1$, we have
 \[ p = \frac{r+1}{2} = \frac{q(q^{s-1} + \cdots + q + 1)}{2} + 1,\]
 meaning $q$ divides $p - 1$.  Hence, there is an automorphism $\sigma$ of $K$ of order $q$ that fixes $0, p \in \Z/2p\Z$ and partitions the remaining elements into orbits of size $q$.
 
 Next, let $e_1, \dots, e_{s+1}$ be a basis for $V$.  Consider the matrix $M = (m_{i,j})$, where $1 \le i,j \le s+1$, with
 \[ m_{i,j} = \begin{cases}
               1, \text{ if } i \le j \le i + 1,\\
               0, \text{ otherwise.}
              \end{cases}\]
That is

 \[ M = \begin{pmatrix}
                    1      & 1      & 0   & \cdots & 0 \\
                    0      & 1      & 1  &  \ddots   & 0\\
                    
                    \vdots    & \ddots &  \ddots  & \ddots & \vdots \\
                   
                    0    &  \cdots & 0  &  1  & 1  \\
                    0    & \cdots    & \cdots  &  0   &   1 
                   \end{pmatrix}.
\]

Now, an induction argument shows that, for $n \ge 1$,  

\[ M^n = \begin{pmatrix}
                    1      & {n \choose 1}      & {n \choose 2}   & \cdots & {n \choose s} \\
                    0      & 1      & {n \choose 1}  &  \ddots   & {n \choose {s-1}}\\
                    
                    \vdots    & \ddots &  \ddots  & \ddots & \vdots \\
                   
                    0    &  \cdots & 0  &  1  & {n \choose 1}  \\
                    0    & \cdots    & \cdots  &  0   &   1 
                   \end{pmatrix},
\]
where we define ${n \choose t}$ to be $0$ when $t > n$.  Since $s < q$, this shows that the matrix $M$ has multiplicative order $q$.  Moreover, we note that $M$ is already in Jordan Canonical Form, and, indeed, if $x = (x_1, \dots, x_{s+1})$ is an eigenvector for $M$, then 
\[(\lambda x_1, \dots, \lambda x_{s+1}) = \lambda x = xM = (x_1, x_1 + x_2, x_2 + x_3, \dots, x_s + x_{s+1}),\]
then $\lambda = 1$ and $x_1 = \cdots = x_s = 0$.  That is, $M$ fixes a unique $1$-dimensional subspace of $V$, and hence $M$ fixes a unique hyperplane $H_1$ of $V$ (and the remaining $r-1$ hyperplanes are partitioned into orbits of length $q$).

We now define an automorphism $\phi$ of $G$ by 
\[ (u, y)^\phi \colonequals (uM, y^\sigma).\] 
It is clear from the construction that $\phi$ has order $q$.  Now, we assign nonidentity elements of $K$ to hyperplanes of $E = (V,+)$ as follows: first, let $H_1$ be the unique hyperplane $\langle e_{s+1} \rangle^\perp$ fixed by $M$ and choose $k_1 = p \in K = (\Z/(r+1)\Z, +) = (\Z/2p\Z, +)$. Next, we note that the remaining $r-1$ hyperplanes are partitioned into $\langle M \rangle$-orbits of size $q$ and the remaining $r-1$ nonidentity elements of $K$ are partitioned into $\langle \sigma \rangle$-orbits of size $q$.  We choose representatives for each $\langle M \rangle$-orbit of size $q$ and for each $\langle \sigma \rangle$-orbit of size $q$, and then we choose any bijection we want between the representatives of orbits of hyperplanes and the representatives of orbits of elements of $K$, and if $H_i \leftrightarrow k_i$, then we extend this to a bijection between hyperplanes and nonidentity elements of $K$ by defining $H_iM^j \leftrightarrow k_i^{\sigma^j}$.  Define
\[ D \colonequals \bigcup_{i = 1}^r (H_i, k_i).\]
Then, by Theorem \ref{thm:McFarland}, $D$ is a DS with McFarland parameters, and our choice of $H_i$ and $k_i$ shows that $D$ is invariant under $\phi$, since $(H_1, k_1)^\phi = (H_1, k_1)$ and 
\[(H_i, k_i)^\phi = (H_iM, k_i^\sigma)\] for $2 \le i \le r$.

Using coordinate notation similar to that used in Subsection \ref{subsect:mcfareven}, define 
\[ \calG \colonequals \langle (1, 0, y), (1, e_i, 0), (\phi, e_{s+1}, 0): y \in K, 1 \le i \le s \rangle.\]
Now, if $X \colonequals (1 \times G) \cap \calG$, then it is clear that
\[ X = \langle (1, 0, y), (1, e_i, 0): y \in K, 1 \le i \le s \rangle\]
and $X^\phi = X$;
moreover, $(\phi, e_{s+1}, 0)$ is an order $q$ element, and 
\[ (\phi, e_{s+1}, 0)^{-1}(1 \times X)(\phi, e_{s+1}, 0) = 1 \times X,\]
so, indeed, $1 \times X$ is normal in $\calG$ and $|\calG| = |G|$.  Next, noting that 
\[ G = \bigcup_{i=0}^{q-1} X(ie_{s+1},0),\] we see that
\[ X^{(\phi, e_{s+1}, 0)^i} = X^{(\phi^i, ie_{s+1},0)} = X^{\phi^i}(ie_{s+1}, 0) = X(ie_{s+1},0),\]
and so $\calG$ is indeed transitive on the cosets of $X$ in $G$.  By Theorem \ref{thm:transfer}, $\calG$ contains a DS with McFarland parameters.

Finally, note that 
\[ \calQ = \langle (\phi, e_{s+1},0), (1, e_i, 0) : 1 \le i \le s \rangle\]
is a Sylow $q$-subgroup of $\calG$.  Since $e_1M \neq e_1$, we have
\[ (\phi, e_{s+1},0) \cdot (1, e_1,0) = (\phi, e_1 + e_{s+1}, 0) \neq (\phi, e_1M + e_{s+1}, 0) = (1, e_1,0) \cdot (\phi, e_{s+1},0),\]
and so $\calQ$ is nonabelian.  Moreover, since 
\[ (1, 0, y)^{-1}(\phi, e_{s+1}, 0)(1,0,y) = (\phi, e_{s+1}, (-y)^\sigma + y) \notin \calQ,\]
we see that $\calG$ does not contain a normal Sylow $q$-subgroup.
\end{proof}

\begin{rem}
 There are $78498$ primes $q$ less than $1000000$, and $q^2 + q + 2 = 2p$, where $p$ is prime, in $5985$ cases, i.e., $5985$ primes less than $1000000$ satisfy the hypotheses of Theorem \ref{thm:oddMcF} with $s = 2$.
\end{rem}

\section{RDSs in nonabelian $2$-groups}
\label{sect:RDS}

Let $q = 2^d$, where $d$ is a positive integer, and define $V \colonequals \GF(q^2)^2.$  Note that $V$ contains $q^2+1$ $\GF(q^2)$-hyperplanes ($1$-dimensional $\GF(q^2)$-subspaces) $H_0, \dots, H_{q^2}$.  Define 
\[ G \colonequals \langle w \rangle \times (V,+) \cong C_{q^2} \times C_2^{4d}.\]  

Consider the automorphism $\phi: G \to G$ defined by
\[ (w^i, (x,y))^\phi \colonequals (w^{-i}, (x^q, y^q));\]
that is, $\phi$ acts as inversion on $\langle w \rangle$ and as a power of the Frobenius map on each coordinate of the elements of $V$.  Clearly, $\phi$ is an order $2$ automorphism.  Note that inversion has $q^2/2 + 1$ orbits on $\langle w \rangle$, namely, $\{1\}$, $\{w^{q^2/2}\}$, and $q^2/2 - 1$ orbits of the form $\{w^i, w^{-i}\}$, $1 \le i \le q^2/2 - 1$.  Moreover, the map $(x,y) \mapsto (x^q, y^q)$ has $q^2/2 + 2$ orbits on hyperplanes, namely, $\{\langle (1,0) \rangle\}$, $\{\langle (0,1) \rangle\}$, $\{\langle (1,1) \rangle\}$, and $q^2/2 - 1$ orbits of the form $\{\langle (1,\alpha^i) \rangle, \langle (1, (\alpha^i)^q) \rangle \},$ where $\alpha$ is a primitive element of $\GF(q^2)$. 

We enumerate the hyperplanes as follows: we define $H_0 \colonequals \langle (1,0) \rangle$, $H_{q^2/2} \colonequals \langle (0,1) \rangle$, $H_{q^2} \colonequals \langle (1,1) \rangle$, and, for the remaining hyperplanes, we assume that $\{H_i, H_{q^2-i}\}$ is an orbit of length two under the map $(x,y) \mapsto (x^q, y^q)$ for each $i$, $1 \le i \le q^2/2 - 1$.

We view the $H_i$ as subgroups of $(V,+)$, and define
\[ R \colonequals \bigcup_{i = 1}^{q^2} (w^i, H_i).\]
Then, $R$ is a $(q^4, q^2, q^4, q^2)$-RDS with respect to the forbidden subgroup $(1, H_0)$.  Moreover, by construction,
\[ (w^{q^2/2}, H_{q^2/2})^\phi = (w^{q^2/2}, H_{q^2/2}),\]
\[ (w^{q^2}, H_{q^2})^\phi = (1, H_{q^2})^\phi = (1, H_{q^2}),\] and, for each $1 \le i \le q^2/2 - 1$, 
\[ (w^i, H_i)^\phi = (w^{q^2 - i}, H_{q^2 - i}),\]
and so $R^\phi = R$.

For ease of notation, we write elements in $\langle \phi \rangle \ltimes G$ in the form $(a, b, c, d)$, where $a \in \langle \phi \rangle$, $b \in \langle w \rangle$, and $(c,d) \in V$.  (We will use similar notation when referring to elements of $G$.)

\begin{thm}
 \label{thm:RDS2gp1}
 Let $q = 2^d$, $V = \GF(q^2)^2$, $\langle w \rangle \cong C_{q^2}$, and $G \cong \langle w \rangle \times (V,+)$.  Then,
 \[ \calG \colonequals \langle (\phi, w, 0,0), (1, w^2,0,0), (1, 1, s, t): (s,t) \in V \rangle \]
 is a nonabelian group containing a $(q^4, q^2, q^4, q^2)$-RDS. 
\end{thm}

\begin{proof}
 Let $\calG$ be as in the statement.  First, 
 \[ \calG \cap (1 \times G) = 1 \times X,\]
 where $X = \langle (w^2, 0,0), (1, s, t) : (s,t) \in V \rangle$.  It is clear that $|X| = |G|/2$, and, since 
 \[ (\phi, w, 0,0)^2 = (\phi^2, w^{-1}w, 0, 0)= (1, 1, 0,0)\in 1 \times X,\] we have $|\calG| = |G|$ and $1 \times X$ is normal in $\calG$, and we see that $X^\phi = X$ from the definition of $\phi$.
 Moreover, since $G = X \cup X(w, 0,0)$, we have
 \[ X^{(\phi, w, 0,0)} = X^\phi(w, 0,0) = X(w,0,0),\]
 and hence the action of $\calG$ on right cosets of $X$ in $G$ is transitive.  Therefore, by Theorem \ref{thm:RDStransfer}, $\calG$ contains an RDS $\calR$ with the same parameters as $R$.
 
 Finally, take $\alpha \in \GF(q^2) \backslash \GF(q)$, which implies $\alpha^q \neq \alpha$.  Since
 \[ (\phi, w, 0,0)(1, 1, \alpha,0) = (\phi, w, \alpha, 0)\]
 whereas
 \[ (1, 1, \alpha,0)(\phi, w, 0,0) = (\phi, w, \alpha^q,0),\]
 we see that $\calG$ is nonabelian, completing the proof.
\end{proof}

Define $\tr: \GF(q^2) \to \GF(q)$ to be the trace map, so $\tr(\alpha) = \alpha^q + \alpha$, and, given $t \in V$, define $B(t) \subset V$ to be a subset such that $B(t) \cup \{t\}$ is a $\GF(2)$-basis for $V$. 

\begin{thm}
 \label{thm:RDS2gp2}
 Let $q = 2^d$, $V = \GF(q^2)^2$, $\langle w \rangle \cong C_{q^2}$, and $G \cong \langle w \rangle \times (V,+)$.  If $\beta \in \GF(q^2) \backslash \GF(q)$, then
 \[ \calG \colonequals \langle (\phi, 1, \beta ,0), (1, w,0,0), (1, 1, s, t)): (s,t) \in B((\beta,0)) \rangle \]
 is a nonabelian group containing a $(q^4, q^2, q^4, q^2)$-RDS. 
\end{thm}

\begin{proof}
 Let $\calG$ be as in the statement, and let $B = B((\beta,0))$..  Note that, since $\beta \notin \GF(q)$ and $B \cup \{(\beta,0)\}$ is a $\GF(2)$-basis for $V$, the subspace generated by $B$ contains all elements of $V$ of the form $(\alpha, 0)$, where $\alpha \in \GF(q)$.  
 
 First, 
 \[ \calG \cap (1 \times G) = 1 \times X,\]
 where $X = \langle (w, 0,0), (1, s, t) : (s,t) \in B \rangle$.  It is clear that $|X| = |G|/2$, and, since $\tr(\beta) \in \GF(q)$, we have 
 \[ (\phi, 1, \beta,0)^2 = (\phi^2, 1, \beta^q + \beta, 0)= (1, 1, \tr(\beta),0)\in 1 \times X,\] and so $|\calG| = |G|$, $1 \times X$ is normal in $\calG$. 
 Now, assume $(1, \gamma, 0) \notin X$.  Since $\gamma = \gamma^{q^2} = (\gamma^q)^q$ and $(1, \gamma^q + \gamma, 0) \in X$ (since $\gamma^q + \gamma \in \GF(q)$, we have $(1, \gamma^q, 0) \notin X$ as well.  This shows that $X^\phi = X$.  Moreover, since $G = X \cup X(1, \beta,0)$, we have
 \[ X^{(\phi, 1, \beta,0)} = X^\phi(1, \beta,0) = X(1,\beta,0),\]
 and hence the action of $\calG$ on right cosets of $X$ in $G$ is transitive.  Therefore, by Theorem \ref{thm:RDStransfer}, $\calG$ contains an RDS $\calR$ with the same parameters as $R$.
 
 Finally, since
 \[ (\phi, 1, \beta,0)(1, w, 0,0) = (\phi, w, \beta,0)\]
 whereas
 \[ (1, w, 0,0)(\phi, 1, \beta,0) = (\phi, w^{-1}, \beta, 0),\]
 we see that $\calG$ is nonabelian, completing the proof.
\end{proof}

\begin{rem}
 In Theorem \ref{thm:RDS2gp2}, we have $(\phi, 1, \beta, 0) \in \calU$, the forbidden subgroup in $\calG$ with respect to $\calR$ (see the proof of Theorem \ref{thm:RDStransfer}).   In particular, we have
 \[ (\phi, 1, \beta, 0)^2 = (1, 1, \tr(\beta), 0),\]
 which shows that $(\phi, 1, \beta, 0)$ has order $4$; that is, the forbidden subgroup $\calU$ is not elementary abelian.  In fact, if $q > 2$, then, as $\GF(2)$-vector spaces, $\dim(\GF(q^2)) > \dim(\GF(q)) + 1$, so there exists $\gamma \in \GF(q^2) \backslash \GF(q)$ such that $(1, 1, \gamma, 0) \in (1 \times X) \cap \calU$ and $\gamma^q \neq \gamma$. Thus, we have
 \[ (\phi, 1, \beta,0)(1, 1, \gamma,0) = (\phi, 1, \beta + \gamma,0)\]
 whereas 
 \[ (1, 1, \gamma,0)(\phi, 1, \beta,0) = (\phi, 1, \beta + \gamma^q, 0),\]
 so the forbidden subgroup $\calU$ is nonabelian.  
\end{rem} 
 
\begin{rem} 
 Relative difference sets with a nonabelian forbidden subgroup were constructed by Hiramine \cite[Theorem 4.8]{Hiramine_2016}; it is possible that the RDSs constructed in Theorems \ref{thm:RDS2gp1}, \ref{thm:RDS2gp2} could be constructed using \cite[Theorem 4.8]{Hiramine_2016}, but it is not immediately clear to the authors.
\end{rem}

\section{Future Work}
\label{sect:future}

We mention here a number of potential future directions.

\begin{itemize}
 \item[(1)]  It seems likely to the authors that Theorem \ref{thm:transfer} can be applied in settings in which the (P)DS is constructed using hyperplanes and coset representatives, such as those in the Hadamard \cite{Wilson_Xiang_1997}, Davis-Jedwab \cite{Davis_Jedwab_1997}, and Chen \cite{Chen_1997} families.  For example, computation in GAP \cite{GAP4} has shown that there are $(320,88,24)$-DSs in the Davis-Jedwab family in SmallGroup(320, 1142), SmallGroup(320, 1146), SmallGroup(320, 1151), SmallGroup(320, 1463), SmallGroup(320, 1511), and SmallGroup(320,1512), so the authors believe that an infinite family can be obtained in nonabelian groups using Theorem \ref{thm:transfer}.
 \item[(2)]  As noted in Remark \ref{rem:mult}, since PDSs in abelian groups obey Ma's Multiplier Theorem (Lemma \ref{lem:multiplier}). Determine the extent to which such automorphisms can be used in conjunction with Theorem \ref{thm:transfer} to construct new examples.
 \item[(3)] Other than the converse to Dillon's Dihedral Trick (Theorem \ref{thm:DillonDihedralConverse}), the constructions of PDSs in this paper were in groups of prime-power order.   Use Theorem \ref{thm:transfer} to generate PDSs -- such as PDSs with Paley parameters -- in groups that do not have prime-power order. 
 \item[(4)]  Theorem \ref{thm:RDStransfer} was only used in Section \ref{sect:RDS}, although there are a number of RDS constructions (see, for example, \cite{Ma_Schmidt_2000}).  While it is known that RDSs exist in certain nonabelian groups (see \cite{Feng_Xiang_2008, Hiramine_2016}), it would be interesting to know the extent to which Theorem \ref{thm:RDStransfer} can be used to construct new examples.  For example, the authors suggest examining RDSs with parameters $(p^{2p}, p^p, p^{2p}, p^p)$, where $p$ is prime.    
 \item[(5)]  We used Theorem \ref{thm:transfer} to construct several infinite families, but these examples are not necessarily the only infinite families for these (P)DSs that can be constructed this way.  It would be quite interesting to see the extent to which Theorem \ref{thm:transfer} can used, and a classification result in the style of \cite{Feng_Li_2021} for any of these families would be especially interesting.  
 \item[(6)]  While Theorem \ref{thm:transfer} provides a method for constructing new examples, it is not necessarily exhaustive.  In order to guaranteed that all examples can be constructed using Theorem \ref{thm:transfer}, it must be the case that $\Aut(\Cay(G,D)) \cong \Aut(G)_D \ltimes G$; however, this is not always the case.  For example, SmallGroup(100,11) and SmallGroup(100,12) each contain a $(100, 36, 14, 12)$-PDS that give rise to the same SRG (see \cite{Jorgensen_Klin_2003}), whereas the full automorphism group of this Cayley graph is isomorphic to $J_2.2$ and does not contain a normal subgroup of order 100.  It is worth noting, however, that computation in GAP \cite{GAP4} shows that one can obtain SmallGroup(100,11) by applying Theorem \ref{thm:transfer} to SmallGroup(100,12) (and vice versa).  It would be interesting to know of any cases where Theorem \ref{thm:transfer} fails to construct all examples.
 \item[(7)] There has been significant interest in related combinatorial objects such as \textit{external difference families}, \textit{external partial difference families}, and \textit{disjoint partial difference families} in recent years due to applications in information security and cryptography; see, for instance \cite{Huczynska_Johnson_2024}.  Is it possible to prove results analogous to Theorem \ref{thm:transfer} and \ref{thm:RDStransfer} in such settings?
 \item[(8)] In recent work \cite{Applebaum_etal_2023}, the groups of order 256 containing a Hadamard DS were classified.  The next open case among $2$-groups is order 1024.  It would be interesting to see which groups can be guaranteed to contain a Hadamard DS using Theorem \ref{thm:transfer}. 
\end{itemize}

\bibliographystyle{plainurl}
\bibliography{NewTransfer}

\begin{thebibliography}{10}

\bibitem{Applebaum_etal_2023}
T.~Applebaum, J.~Clikeman, J.~A. Davis, J.~F. Dillon, J.~Jedwab, T.~Rabbani,
  K.~Smith, and W.~Yolland.
\newblock Constructions of difference sets in nonabelian 2-groups.
\newblock {\em Algebra Number Theory}, 17(2):359--396, 2023.
\newblock \href {https://doi.org/10.2140/ant.2023.17.359}
  {\path{doi:10.2140/ant.2023.17.359}}.

\bibitem{Ball_2015}
Simeon Ball.
\newblock {\em Finite geometry and combinatorial applications}, volume~82 of
  {\em London Mathematical Society Student Texts}.
\newblock Cambridge University Press, Cambridge, 2015.
\newblock \href {https://doi.org/10.1017/CBO9781316257449}
  {\path{doi:10.1017/CBO9781316257449}}.

\bibitem{Bao_etal_2024}
Jingjun Bao, Qing Xiang, and Meng Zhao.
\newblock Partial difference sets with denniston parameters in elementary
  abelian $p$-groups.
\newblock \url{http://arxiv.org/abs/2407.15632v1}, 2024.

\bibitem{Beth_Jungnickel_Lenz_1999}
Thomas Beth, Dieter Jungnickel, and Hanfried Lenz.
\newblock {\em Design theory. {V}ol. {I}}, volume~69 of {\em Encyclopedia of
  Mathematics and its Applications}.
\newblock Cambridge University Press, Cambridge, second edition, 1999.
\newblock \href {https://doi.org/10.1017/CBO9780511549533}
  {\path{doi:10.1017/CBO9780511549533}}.

\bibitem{Chen_1997}
Yu~Qing Chen.
\newblock On the existence of abelian {H}adamard difference sets and a new
  family of difference sets.
\newblock {\em Finite Fields Appl.}, 3(3):234--256, 1997.
\newblock \href {https://doi.org/10.1006/ffta.1997.0184}
  {\path{doi:10.1006/ffta.1997.0184}}.

\bibitem{Davis_etal_2023b}
James Davis, Sophie Huczynska, Laura Johnson, and John Polhill.
\newblock Denniston partial difference sets exist in the odd prime case.
\newblock \url{http://arxiv.org/abs/2311.00512v2}, 2023.

\bibitem{Davis_etal_2023}
James Davis, John Polhill, Ken Smith, and Eric Swartz.
\newblock Nonabelian partial difference sets constructed using abelian
  techniques.
\newblock \url{https://arxiv.org/abs/2307.15648}, 2023.

\bibitem{Davis_etal_2024}
James Davis, John Polhill, Ken Smith, Eric Swartz, and Jordan Webster.
\newblock New spence difference sets.
\newblock {\em Des. Codes Cryptogr.}, page to appear.

\bibitem{Davis_Jedwab_1997}
James~A. Davis and Jonathan Jedwab.
\newblock A unifying construction for difference sets.
\newblock {\em J. Combin. Theory Ser. A}, 80(1):13--78, 1997.
\newblock \href {https://doi.org/10.1006/jcta.1997.2796}
  {\path{doi:10.1006/jcta.1997.2796}}.

\bibitem{Davis_Xiang_2000}
James~A. Davis and Qing Xiang.
\newblock A family of partial difference sets with {D}enniston parameters in
  nonelementary abelian 2-groups.
\newblock {\em European J. Combin.}, 21(8):981--988, 2000.
\newblock \href {https://doi.org/10.1006/eujc.2000.0416}
  {\path{doi:10.1006/eujc.2000.0416}}.

\bibitem{DeWinter_2023}
Stefaan De~Winter.
\newblock Projective two-weight sets of denniston type.
\newblock \url{https://arxiv.org/abs/2311.00827}, 2023.

\bibitem{Denniston_1969}
R.~H.~F. Denniston.
\newblock Some maximal arcs in finite projective planes.
\newblock {\em J. Combinatorial Theory}, 6:317--319, 1969.

\bibitem{Dillon_1985}
J.~F. Dillon.
\newblock Variations on a scheme of {M}c{F}arland for noncyclic difference
  sets.
\newblock {\em J. Combin. Theory Ser. A}, 40(1):9--21, 1985.
\newblock \href {https://doi.org/10.1016/0097-3165(85)90043-3}
  {\path{doi:10.1016/0097-3165(85)90043-3}}.

\bibitem{Drisko_1998}
Arthur~A. Drisko.
\newblock Transversals in row-{L}atin rectangles.
\newblock {\em J. Combin. Theory Ser. A}, 84(2):181--195, 1998.
\newblock \href {https://doi.org/10.1006/jcta.1998.2894}
  {\path{doi:10.1006/jcta.1998.2894}}.

\bibitem{Feng_He_Chen_2020}
Tao Feng, Zhiwen He, and Yu~Qing Chen.
\newblock Partial difference sets and amorphic {C}ayley schemes in non-abelian
  2-groups.
\newblock {\em J. Combin. Des.}, 28(4):273--293, 2020.
\newblock \href {https://doi.org/10.1002/jcd.21695}
  {\path{doi:10.1002/jcd.21695}}.

\bibitem{Feng_Li_2021}
Tao Feng and Weicong Li.
\newblock The point regular automorphism groups of the {P}ayne derived
  quadrangle of {$W(q)$}.
\newblock {\em J. Combin. Theory Ser. A}, 179:Paper No. 105384, 53, 2021.
\newblock \href {https://doi.org/10.1016/j.jcta.2020.105384}
  {\path{doi:10.1016/j.jcta.2020.105384}}.

\bibitem{Feng_Xiang_2008}
Tao Feng and Qing Xiang.
\newblock Semi-regular relative difference sets with large forbidden subgroups.
\newblock {\em J. Combin. Theory Ser. A}, 115(8):1456--1473, 2008.
\newblock \href {https://doi.org/10.1016/j.jcta.2008.02.003}
  {\path{doi:10.1016/j.jcta.2008.02.003}}.

\bibitem{GAP4}
The GAP~Group.
\newblock {\em {GAP -- Groups, Algorithms, and Programming, Version 4.12.2}},
  2022.
\newblock URL: \url{https://www.gap-system.org}.

\bibitem{Godsil_Roy_2009}
Chris Godsil and Aidan Roy.
\newblock Equiangular lines, mutually unbiased bases, and spin models.
\newblock {\em European J. Combin.}, 30(1):246--262, 2009.
\newblock \href {https://doi.org/10.1016/j.ejc.2008.01.002}
  {\path{doi:10.1016/j.ejc.2008.01.002}}.

\bibitem{Hiramine_2016}
Yutaka Hiramine.
\newblock On automorphism groups of divisible designs acting regularly on the
  set of point classes.
\newblock {\em Des. Codes Cryptogr.}, 79(2):319--335, 2016.
\newblock \href {https://doi.org/10.1007/s10623-015-0054-x}
  {\path{doi:10.1007/s10623-015-0054-x}}.

\bibitem{Huczynska_Johnson_2024}
Sophie Huczynska and Laura Johnson.
\newblock New constructions for disjoint partial difference families and
  external partial difference families.
\newblock {\em J. Combin. Des.}, 32(4):190--213, 2024.
\newblock \href {https://doi.org/10.1002/jcd.21930}
  {\path{doi:10.1002/jcd.21930}}.

\bibitem{Jorgensen_Klin_2003}
L.~K. J{\o}rgensen and M.~Klin.
\newblock Switching of edges in strongly regular graphs. {I}. {A} family of
  partial difference sets on 100 vertices.
\newblock {\em Electron. J. Combin.}, 10:Research Paper 17, 31, 2003.
\newblock \href {https://doi.org/10.37236/1710} {\path{doi:10.37236/1710}}.

\bibitem{Leifman_Muzychuk_2005}
Yefim~I. Leifman and Mikhail~E. Muzychuk.
\newblock Strongly regular {C}ayley graphs over the group {$\Bbb {Z}_{p^n}
  \oplus \Bbb {Z}_{p^n}$}.
\newblock {\em Discrete Math.}, 305(1-3):219--239, 2005.
\newblock \href {https://doi.org/10.1016/j.disc.2005.06.027}
  {\path{doi:10.1016/j.disc.2005.06.027}}.

\bibitem{Ma_1984}
S.~L. Ma.
\newblock Partial difference sets.
\newblock {\em Discrete Math.}, 52(1):75--89, 1984.
\newblock \href {https://doi.org/10.1016/0012-365X(84)90105-5}
  {\path{doi:10.1016/0012-365X(84)90105-5}}.

\bibitem{Ma_1994b}
S.~L. Ma.
\newblock A survey of partial difference sets.
\newblock {\em Des. Codes Cryptogr.}, 4(3):221--261, 1994.
\newblock \href {https://doi.org/10.1007/BF01388454}
  {\path{doi:10.1007/BF01388454}}.

\bibitem{Ma_Schmidt_2000}
Siu~Lun Ma and Bernhard Schmidt.
\newblock Relative {$(p^a,p^b,p^a,p^{a-b})$}-difference sets: a unified
  exponent bound and a local ring construction.
\newblock {\em Finite Fields Appl.}, 6(1):1--22, 2000.
\newblock \href {https://doi.org/10.1006/ffta.1999.0254}
  {\path{doi:10.1006/ffta.1999.0254}}.

\bibitem{Malandro_Smith_2020}
Martin~E. Malandro and Ken~W. Smith.
\newblock Partial difference sets in {$C_{2^n}\times C_{2^n}$}.
\newblock {\em Discrete Math.}, 343(4):111744, 22, 2020.
\newblock \href {https://doi.org/10.1016/j.disc.2019.111744}
  {\path{doi:10.1016/j.disc.2019.111744}}.

\bibitem{McFarland_1973}
Robert~L. McFarland.
\newblock A family of difference sets in non-cyclic groups.
\newblock {\em J. Combinatorial Theory Ser. A}, 15:1--10, 1973.
\newblock \href {https://doi.org/10.1016/0097-3165(73)90031-9}
  {\path{doi:10.1016/0097-3165(73)90031-9}}.

\bibitem{Moore_Pollatsek_2013}
Emily~H. Moore and Harriet~S. Pollatsek.
\newblock {\em Difference sets}, volume~67 of {\em Student Mathematical
  Library}.
\newblock American Mathematical Society, Providence, RI, 2013.
\newblock Connecting algebra, combinatorics, and geometry.
\newblock \href {https://doi.org/10.1090/stml/067}
  {\path{doi:10.1090/stml/067}}.

\bibitem{Polhill_etal_2023}
John Polhill, James Davis, Kenneth Smith, and Eric Swartz.
\newblock Genuinely nonabelian partial difference sets.
\newblock {\em J. Combin. Des.}, 32(7):351--370, 2024.
\newblock \href {https://doi.org/10.1002/jcd.21938}
  {\path{doi:10.1002/jcd.21938}}.

\bibitem{Pott_etal_1999}
A.~Pott, P.~V. Kumar, T.~Helleseth, and D.~Jungnickel, editors.
\newblock {\em Difference sets, sequences and their correlation properties},
  volume 542 of {\em NATO Advanced Science Institutes Series C: Mathematical
  and Physical Sciences}. Kluwer Academic Publishers Group, Dordrecht, 1999.
\newblock \href {https://doi.org/10.1007/978-94-011-4459-9}
  {\path{doi:10.1007/978-94-011-4459-9}}.

\bibitem{Skinner_1988}
Gerald~K. Skinner.
\newblock X-ray imaging with coded masks.
\newblock {\em Scientific American}, 259(2):84--89, 1988.
\newblock URL: \url{http://www.jstor.org/stable/24989198}.

\bibitem{Spence_1977}
Edward Spence.
\newblock A family of difference sets.
\newblock {\em J. Combinatorial Theory Ser. A}, 22(1):103--106, 1977.
\newblock \href {https://doi.org/10.1016/0097-3165(77)90068-1}
  {\path{doi:10.1016/0097-3165(77)90068-1}}.

\bibitem{Swartz_2015}
Eric Swartz.
\newblock A construction of a partial difference set in the extraspecial groups
  of order {$p^3$} with exponent {$p^2$}.
\newblock {\em Des. Codes Cryptogr.}, 75(2):237--242, 2015.
\newblock \href {https://doi.org/10.1007/s10623-013-9903-7}
  {\path{doi:10.1007/s10623-013-9903-7}}.

\bibitem{Tao_Feng_Li_2021}
Ran Tao, Tao Feng, and Weicong Li.
\newblock A construction of minimal linear codes from partial difference sets.
\newblock {\em IEEE Trans. Inform. Theory}, 67(6):3724--3734, 2021.
\newblock \href {https://doi.org/10.1109/TIT.2021.3067049}
  {\path{doi:10.1109/TIT.2021.3067049}}.

\bibitem{Turyn_1965}
Richard~J. Turyn.
\newblock Character sums and difference sets.
\newblock {\em Pacific J. Math.}, 15:319--346, 1965.
\newblock URL: \url{http://projecteuclid.org/euclid.pjm/1102996015}.

\bibitem{Wilson_Xiang_1997}
Richard~M. Wilson and Qing Xiang.
\newblock Constructions of {H}adamard difference sets.
\newblock {\em J. Combin. Theory Ser. A}, 77(1):148--160, 1997.
\newblock \href {https://doi.org/10.1006/jcta.1996.2740}
  {\path{doi:10.1006/jcta.1996.2740}}.

\bibitem{Xia_etal_2005}
Pengfei Xia, Shengli Zhou, and Georgios~B. Giannakis.
\newblock Achieving the {W}elch bound with difference sets.
\newblock {\em IEEE Trans. Inform. Theory}, 51(5):1900--1907, 2005.
\newblock \href {https://doi.org/10.1109/TIT.2005.846411}
  {\path{doi:10.1109/TIT.2005.846411}}.

\bibitem{Xia_etal_2007}
Pengfei Xia, Shengli Zhou, and Georgios~B. Giannakis.
\newblock Correction to: ``{A}chieving the {W}elch bound with difference sets''
  [{IEEE} {T}rans. {I}nform. {T}heory {\bf 51} (2005), no. 5, 1900--1907;
  mr2235693].
\newblock {\em IEEE Trans. Inform. Theory}, 52(7):3359, 2006.
\newblock \href {https://doi.org/10.1109/TIT.2006.876214}
  {\path{doi:10.1109/TIT.2006.876214}}.

\end{thebibliography}

\end{document}